\documentclass[reqno,12pt,twosided,a4paper]{amsart}
\usepackage{amssymb,amscd,stmaryrd}
\usepackage{amsmath}
\usepackage{amsfonts}
\usepackage[top=3cm,bottom=3cm,left=2.5cm,right=2.5cm,footskip=17mm]{geometry}
\usepackage[mathcal]{eucal}
\usepackage{array,float}
\usepackage{enumitem}
\usepackage[OT2,T1]{fontenc}
\usepackage[russian,USenglish]{babel}
\usepackage{chngcntr}
\usepackage{apptools}
\AtAppendix{\counterwithin{theorem}{section}}
\usepackage{float}
\newfloat{Diagram}{htbp}{dia}

\usepackage{amsmath}
\usepackage{amssymb, xy}
\input xy
\xyoption{all}
\usepackage{pdflscape}
\usepackage{hyperref}
\usepackage[draft]{graphicx}
\usepackage{tikz,pgf}
\usepackage{tikzit}

\tikzstyle{1function}=[fill=white, draw=black, shape=rectangle, minimum width=0.75cm, minimum height=1 cm]
\tikzstyle{2function}=[fill=white, draw=black, shape=rectangle, minimum width=1cm, minimum height=1 cm]
\tikzstyle{3function}=[fill=white, draw=black, shape=rectangle, minimum width=1.5cm, minimum height=1cm]
\tikzstyle{multi function}=[fill=white, draw=black, shape=rectangle, minimum width=5cm, minimum height=1cm]

\usepackage{xcolor}
\usepackage{color}

\numberwithin{equation}{section}

\newtheorem{theorem}{Theorem}[section]
\newtheorem{lemma}[theorem]{Lemma}
\newtheorem{proposition}[theorem]{Proposition}
\newtheorem{corollary}[theorem]{Corollary}

\theoremstyle{definition}
\newtheorem{definition}[theorem]{Definition}

\newtheorem{question}[theorem]{Question}

\newtheorem{remark}[theorem]{Remark}

\newcommand{\udi}[1]{{\color{blue}{#1}}}
\newcommand{\ot}{\otimes}

\newcommand{\N}{\mathbb{N}}

\newcommand{\one}{\textbf{1}}

\newcommand{\Z}{\mathbb{Z}}
\newcommand{\Ga}{\Gamma}
\newcommand{\Ker}{\text{Ker}}
\newcommand{\End}{\text{End}}
\newcommand{\Hom}{\text{Hom}}
\newcommand{\M}{\text{M}}
\newcommand{\parti}{\vdash}
\newcommand{\la}{\lambda}
\newcommand{\ol}{\overline}
\renewcommand{\S}{\mathbb{S}}
\newcommand{\GL}{\text{GL}}
\newcommand{\R}{\text{R}}
\newcommand{\Ind}{\text{Ind}}
\newcommand{\Q}{\mathbb{Q}}

\newcommand{\Res}{\text{Res}}
\newcommand{\Ainf}{K[X]}
\newcommand{\Delot}{\Delta^{\otimes}}
\newcommand{\Tr}{\text{Tr}}
\newcommand{\Id}{\text{Id}}
\newcommand{\T}{\mathcal{T}}
\newcommand{\Pp}{\text{P}}

\newcommand{\winf}{K[X]_{aug}}

\renewcommand{\Re}{\text{Re}}

\newcommand{\pair}{\text{pair}}
\newcommand{\Aut}{\text{Aut}}
\newcommand{\laa}{\{\lambda\}}
\newcommand{\muu}{\{\mu\}}

\newcommand{\I}{\mathfrak{I}}
\newcommand{\xini}{x_1^{\ot n_1}\ot\cdots\ot x_r^{\ot n_r}}

\title{Universal rings of invariants}
\author{Ehud Meir}
\address{Institute of Mathematics, University of Aberdeen, Fraser Noble Building, Aberdeen AB24 3UE, UK}
  \email{meirehud@gmail.com}

\begin{document}
\maketitle
\begin{abstract}
Let $K$ be an algebraically closed field of characteristic zero. 
Algebraic structures of a specific type (e.g. algebras or coalgebras) on a given vector space $W$ over $K$ can be encoded as points in an affine space $U(W)$. This space is equipped with a $\GL(W)$ action, and two points define isomorphic structures if and only if they lie in the same orbit. This leads to study the ring of invariants $K[U(W)]^{\GL(W)}$. We describe this ring by generators and relations. We then construct combinatorially a commutative ring $\Ainf$ which specializes to all rings of invariants of the form $K[U(W)]^{\GL(W)}$. We show that the commutative ring $\Ainf$ has a richer structure of a Hopf algebra with additional coproduct, grading, and an inner product which makes it into a rational PSH-algebra, generalizing a structure introduced by Zelevinsky. We finish with a detailed study of $\Ainf$ in the case of an algebraic structure consisting of a single endomorphism, and show how the rings of invariants $K[U(W)]^{\GL(W)}$ can be calculated explicitly from $\Ainf$ in this case.  
\end{abstract}

\section{Introduction}
In \cite{Procesi} Procesi studied tuples $(T_1,\ldots, T_r)$ of endomorphisms of a finite dimensional vector space up to simultaneous conjugation by studying the corresponding ring of invariants. He showed that for a field $K$ of characteristic zero the algebra of invariants $C_{d,r}=K[(T_i)_{j,k}]^{\GL_d(K)}$ can be generated by traces of monomials in $T_1,\ldots, T_r$, and that all the polynomial relations among these traces arise from the Cayley-Hamilton Theorem. This gives an infinite presentation for the algebra of invariants $C_{d,r}$. This algebra, however, is known to be finitely generated by a theorem of Nagata. 
Describing explicitly a finite presentation for this algebra when $r>1$ is a difficult task. 
See \cite{Teranishi}, \cite{Nakamoto}, \cite{ADS}, \cite{BD} and \cite{Hoge} for results about such presentations for $C_{3,r}$ and $C_{2,r}$ for various values of $r$. 


The main tool used in the work of Procesi to describe the invariants is Schur-Weyl duality. This tool was used further to study more complicated algebraic structures than a vector space equipped with endomorphisms. In \cite{DKS} and \cite{meir1}  Datt Kodiyalam and Sunder and the author of this paper applied invariant theory to the study of finite dimensional semisimple Hopf algebras. In the second paper invariant theory was also used to prove that a semisimple Hopf algebra admits at most finitely many Hopf orders over a number field. Invariant theory was also applied to other Hopf-algebra-related structures such as Hopf two-cocycles and Nichols algebras in \cite{meir2} and \cite{meir4}. See also \cite{meir5} for the study of linear endomorphisms of the tensor product of vector spaces using invariant theory. 

The first goal of the present paper is to generalize the above results to any type of algebraic structure based on a finite dimensional vector space over an algebraically closed field of characteristic zero. The second goal of this paper is to introduce a uniform approach for studying these rings of invariants, which does not depends on the dimension of the vector space. We will show that this results in an infinitely generated polynomial algebra which specializes to invariant rings arising in all possible dimensions. This universal invariant ring, which we will denote by $\Ainf$, is further equipped with a natural pairing and a self-dual Hopf algebra structure. This kind of structure resembles in many ways the PSH-algebras introduced by Zelevinsky in \cite{Zelevinsky}. 

To give the precise details, let $K$ be an algebraically closed field of characteristic zero. We define an \emph{algebraic structure} over $K$ to be a finite dimensional vector space $W$ equipped with structure tensors $$x_i\in W^{p_i,q_i} := W^{\ot p_i}\ot (W^*)^{\ot q_i}, \text{ for } i=1,\ldots r.$$ So for example an algebra is described by a tensor in $W^{1,2}$, and a unit element can be thought of as a map $K\to W$, or an element in $W^{1,0}$. The tuple $((p_1,q_1),\ldots,(p_r,q_r))= ((p_i,q_i))\in (\N^2)^r$ will be referred to as the \emph{type} of $W$ and will be fixed throughout the paper.

The set of all possible such structures on $W$ gives rise to a vector space (or an affine space, from the algebraic geometry point of view) $U(W)$. If we are interested only in structures which satisfy some set of axioms $\T$ then we will restrict out attention to the subset $Y(W)\subseteq U(W)$ of all points which satisfy these axioms. In most cases $Y(W)$ is a Zariski closed subset of $U(W)$ (see Section \ref{sec:axid}). The group $\GL(W)$ acts on $U(W)$ and on $Y(W)$, and two points in $Y(W)$ define isomorphic structures if and only if they lie in the same orbit. This leads to study the quotient $Y(W)/\GL(W)$ and its algebraic counterpart $K[Y(W)]^{\GL(W)}$. From here on we will write $Y_d = Y(K^d)$ and $U_d = U(K^d)$ (See Subsection \ref{subsec:algstr}).

Let now $(n_1,\ldots,n_r)\in \N^r$ be a tuple of integers such that $\sum_ip_in_i = \sum_i q_in_i=n$, and let $\sigma\in S_n$ be any permutation.
Define 
$$p(n,\sigma,n_1,\ldots,n_r)((x_i)):= \Tr_{W^{\ot n}}(L^{(n)}_{\sigma}(\xini))$$
where $L^{(n)}_{\sigma}$ is the linear map given by $$w_1\ot\ldots\ot w_n\mapsto w_{\sigma^{-1}(1)}\ot\ldots\ot w_{\sigma^{-1}(n)}.$$
The expression $p(n,\sigma,n_1,\ldots,n_r)$ can be considered as a polynomial in $K[Y_d]^{\GL_d(K)}$. We will refer to such invariants as \emph{basic invariants}.
In Subsection \ref{subsec:invariantalg} we will use Schur-Weyl duality to prove that $K[Y_d]^{\GL_d(K)}$ has the following (infinite) presentation: 
\begin{theorem}\label{thm:mainstructure}
The algebra $K[Y_d]^{\GL_d(K)}$ is generated by the elements $p(n,\sigma,n_1,\ldots,n_r)= \Tr(L^{(n)}_{\sigma}\xini)$. The relations between these elements are the following:
\begin{enumerate}
\item[R0](product) The product of two basic invariants is again a basic invariant (see Proposition \ref{prop:ainfform})
\item[R1](cyclicity of trace) For every $(\mu_i)\in S_{n_1}\times\cdots\times S_{n_r}$ we have 
$$p(n,\sigma,n_1,\ldots,n_r) = p(n,\alpha^{(p_i)}_{(n_i)}(\mu_i)\sigma\alpha^{(q_i)}_{(n_i)}(\mu_i),n_1,\ldots,n_r)$$
(the homomorphisms $\alpha^{(p_i)}_{(n_i)}$ and $\alpha^{(q_i)}_{(n_i)}$ are defined in Subsection \ref{subsec:permutations}).

\item[R2](dimension relations)If $n>d$ and  $\tau_1,\tau_2\in S_n$ then 
$$\sum_{\sigma\in S_{d+1}}(-1)^{\sigma}p(n,\tau_1\sigma\tau_2,n_1,\ldots,n_r)=0.$$
\item[R3] Relations arising from the axioms. If we write $I_{\T,d} = \Ker(K[U_d]\to K[Y_d])$ then these are the polynomials in $I_{\T,d}^{\GL_d(K)}$. 
The letter $\T$ represents here the set of axioms the structures in $Y_d$ satisfy.
\end{enumerate}
\end{theorem}
Relations of type $R0$ are rather simple. We will show that every basic invariant corresponds to a certain diagram, and that the product is given by taking disjoint union of diagrams.
We will consider relations of type $R3$ in Section \ref{sec:axid}. For most of the paper, and for the rest of the introduction, we will concentrate on the invariant ring $K[U_d]^{\GL_d(K)}$. 
By the result of Nagata we know that $K[U_d]^{\GL_d(K)}$ is a finitely generated algebra. However, just like in the case of a vector space with endomorphisms, finding a finite presentation for this algebra is rather difficult.
By geometric invariant theory (see Subsection \ref{subsec:GIT}) we know that $K[U_d]^{\GL_d(K)}$ is the coordinate ring of an affine variety $X_d$, whose points correspond to closed orbits in $Y_d$. We write $$X = \bigsqcup_{d\geq 0} X_d.$$  

The presentation we have for $K[U_d]^{\GL_d(K)}$ is rather uniform in $d$: Indeed, the set of generators does not depend on $d$ and all the relations except those of type $R2$ do not depend on $d$. 
Our approach in this paper will be to study these algebras by studying their ``limit'' $$\lim_{d\to \infty} K[U_d]^{\GL_d(K)}.$$ This will be a commutative algebra which we will denote by $\Ainf$. We will give two equivalent definitions for $\Ainf$. The first one is based on taking out relations of type $R2$ from the presentation we have for $K[U_d]^{\GL_d(K)}$:
\begin{definition}[The algebra $\Ainf$- first definition]\label{def:firstdefkx} 
The algebra $\Ainf$ is the commutative algebra generated by the symbols $p(n,\sigma,n_1,\ldots,n_r)$ modulo the relations $R0$ and $R1$ of Theorem \ref{thm:mainstructure}.
\end{definition}
Notice that every element in $\Ainf$ can be evaluated on every element in $X$. This is the reason for the notation $\Ainf$. 
The second definition for $\Ainf$ will be based on diagrams. 
A \emph{string diagram} is a diagram which represents a linear map between some tensor powers of the algebraic structure $W$. It is made of boxes labeled by $x_1,\ldots, x_r$ and $\Id_W$, and input and output strings. Some of the input strings may be connected to the output strings. We consider diagrams to be equivalent if they represent the same linear map for every algebraic structure (see Definition \ref{def:eqDi}). A diagram with free $q$ input strings and free $p$ output strings represents a linear map from $W^{\ot q}$ to $W^{\ot p}$ and is said to have degree $(p,q)$. The following figure is an example of a diagram of degree $(2,1)$, representing the linear map $ev(L^{(2)}_{(12)}(x_2\ot x_1))$ (see Subsections \ref{subsec:natid} and \ref{subsec:SW} for the relevant terminology)
\begin{center}
\begin{tikzpicture}
	\begin{pgfonlayer}{nodelayer}
		\node [style=1function] (0) at (0, 4.75) {$x_2$};
		\node [style=1function] (1) at (1.25, 4.75) {$x_1$};
		\node [style=none] (2) at (0, 4.25) {};
		\node [style=none] (3) at (0, 5.25) {};
		\node [style=none] (4) at (1.5, 4.25) {};
		\node [style=none] (5) at (1, 4.25) {};
		\node [style=none] (6) at (1.25, 5.25) {};
		\node [style=none] (7) at (0, 5.75) {};
		\node [style=none] (8) at (1.25, 5.75) {};
		\node [style=none] (9) at (0, 3.75) {};
		\node [style=none] (10) at (1.5, 3.75) {};
		\node [style=none] (11) at (1, 3.75) {};
		\node [style=none] (12) at (2.75, 5.75) {};
		\node [style=none] (13) at (2.75, 3.75) {};
	\end{pgfonlayer}
	\begin{pgfonlayer}{edgelayer}
		\draw (2.center) to (9.center);
		\draw (7.center) to (3.center);
		\draw (8.center) to (6.center);
		\draw (5.center) to (11.center);
		\draw (4.center) to (10.center);
		\draw (13.center) to (12.center);
		\draw [bend left=90, looseness=0.75] (7.center) to (12.center);
		\draw [bend left=90, looseness=1.25] (13.center) to (10.center);
	\end{pgfonlayer}
\end{tikzpicture}
\end{center}
A closed diagram is a diagram with no free input or output strings. In other words- it is a diagram of degree $(0,0)$. 
If $Di_1$ and $Di_2$ are two diagrams then we denote by $Di_1\star Di_2$ the diagram resulting from placing $Di_1$ to the left of $Di_2$ (see Section \ref{sec:diagrams}). 
This defines an associative multiplication on the set of all diagrams which is also commutative on the set of closed diagrams.
This leads us the second definition of $\Ainf$: 
\begin{definition}\label{def:seconddefkx}
We define $\Ainf$ to be the free vector space on the set of all equivalence classes of closed diagrams in which no boxes are labeled by $\Id_W$. 
We define $\winf$ to be the free vector space on the set of all equivalence classes of closed diagrams. The $\star$-product of diagrams equips both these vector spaces with an algebra structure.
\end{definition}
We will show in Section \ref{sec:ainfwinf} that the two definitions of $\Ainf$ are indeed equivalent.
The reason for the notation for $\winf$ is that we refer to diagrams which contain $\Id_W$ as augmented, see Definition \ref{def:augmented}. 
Closed diagrams can be thought of as representing linear maps from $W^{\ot 0} = K$ to itself, or in other words, as scalars. 
Since product corresponds to taking disjoint union and every diagram splits uniquely to the disjoint union of its connected components, we have the following result about the structure of $\Ainf$ (see Corollary \ref{cor:ainfwinf}):
\begin{proposition}
The algebra $\Ainf$ is a polynomial algebra on the set of all connected diagrams which do not contain $\Id_W$. 
\end{proposition}
Since the equivalence relation for diagrams becomes more complicated once $\Id_W$ is involved, we should be more careful about a similar statement for $\winf$. We will show that in fact $\winf\cong \Ainf\ot K[D]$ where $D$ is an extra variable corresponding to the dimension of $W$ (see Corollary \ref{cor:ainfwinf}). 

Beyond giving a uniform description for the algebras $K[Y_d]$ and $K[U_d]$ the algebra $\Ainf$ has the advantage of having a much richer structure than just a commutative algebra. We will show in Section \ref{sec:bialg} that the operations of forming direct sums and tensor products of structures induce two coproducts $\Delta,\Delot:\Ainf\to \Ainf\ot\Ainf$ on $\Ainf$. We will prove the following:
\begin{theorem}
The coproduct $\Delta$ equips $\Ainf$ with a Hopf algebra structure. The connected diagrams are primitive with respect to this coproduct. That is- they satisfy the equations $\Delta(p) = p\ot 1 + 1\ot p$ and $\epsilon(p)=0$. 
The coproduct $\Delot$ equips $\Ainf$ with a bialgebra structure. All the diagrams are group like elements with respect to this coproduct. That is- they satisfy the equations $\Delot(p) = p\ot p$ and $\epsilon(p) = 1$. 
\end{theorem}
The algebra $\Ainf$ is graded by $\N^r$, where $(\Ainf)_{n_1,\ldots,n_r}$ is spanned by the diagrams corresponding to basic invariants of the form $p(n,\sigma,n_1,\ldots,n_r)$ where $n=\sum_i n_ip_i=\sum_i n_iq_i$ and $\sigma\in S_n$ (if $\sum_i n_ip_i\neq \sum_i n_iq_i$ then $(\Ainf)_{n_1,\ldots,n_r} =0$). The multiplication and the coproduct $\Delta$ respect this grading (the other comultiplication $\Delot$ does not). Thus, $\Ainf$ has a structure of a graded Hopf algebra. We will show that if $\sum_i n_ip_i = \sum_i n_iq_i = n$ then $(\Ainf)_{n_1,\ldots,n_r}$ is equipped with an inner product $\langle -,-\rangle$ arising from the natural inner product on the group algebra $KS_n$. We will then show that with respect to this inner product, the multiplication is adjoint to the comultiplication. In other words, if we use the Sweedler notation $\Delta(z) = z_1\ot z_2$ then $$\langle z,xy\rangle = \langle z_1,x\rangle\langle z_2,y\rangle.$$
Zelevinsky studied a similar family of Hopf algebras which he called positive self adjoint Hopf algebras (or PSH-algebras).
PSH-algebras are defined over $\Z$ and have a very rigid structure- there is only one ``simple'' (or \emph{universal}) PSH-algebra, which encodes the representation theory of all the symmetric groups, and every other PSH algebra decomposes as a tensor product of copies of this universal PSH-algebra. Our Hopf algebras, however, are defined over a field of characteristic zero and not over $\Z$. In Section \ref{sec:pshalg} we will define \emph{rational} PSH-algebras as a certain generalization of PSH-algebras. We will then show the following:
\begin{theorem}\label{thm:Ainfpsh}
The Hopf algebra $\Ainf$ is a rational PSH-algebra. 
\end{theorem}
In Section \ref{sec:example} we will study in detail the algebra $\Ainf$ for the structure of a vector space with a single endomorphism. We will show that in this case the algebra $\Ainf$ is the extension of scalars from $\Z$ to $K$ of the simple PSH-algebra. This enables us to give a very clean description of the ideals $I_d:= \Ker(\Ainf\to K[U_d]^{\GL_d(K)})$. Indeed- we will show that $\Ainf\cong K[y_1,y_2,\ldots]$ and that $I_d= (y_{d+1},y_{d+2},\ldots,)$. This will recover the well known result that the algebra of invariants $K[\M_d(K)]^{\GL_d(K)}$ is a polynomial ring in $d$ variables. 
Zelevinsky introduced PSH-algebra in order to study representations of families of finite groups, such as $S_n$ or $\GL_n(\mathbb{F}_q)$. This brings us to the following question:
\begin{question}
Does the rational PSH-algebra $\Ainf$ have a natural basis which defines a PSH-algebra over $\Z$? In other words, can we always find inside $\Ainf$ a PSH-algebra $R$ such that $\Ainf\cong R\ot_{\Z} K$? Moreover, is there a representation theoretic interpretation for this PSH-algebra?
\end{question}

In Section \ref{sec:hilbert} we will give a formula for the Hilbert function of the multi-graded algebra $\Ainf$ and its finitely generated quotients $\Ainf/I_d\cong K[U_d]^{\GL_d(K)}$ in terms of the Littlewood-Richardson coefficients and the Kronecker coefficients. A similar concrete calculation was done for the invariant ring of an endomorphism of the tensor product of two vector spaces of dimension 2, see \cite{meir5}. While the formula we get in Section \ref{sec:hilbert} seems to be quite complicated it establishes a connection between our rings of invariants and central themes in the representation theory of the symmetric groups, such as the Kronecker and the Littlewood-Richardson coefficients. 

\section{Preliminaries}\label{sec:prelim}
\subsection{Geometric invariant theory}\label{subsec:GIT}
Recall that a linear algebraic group $\Ga$ is called \emph{reductive} if the category of rational representations of $\Ga$ is semisimple. In other words- every short exact sequence $$0\to V'\to V\to V''\to 0$$ of rational $\Ga$-representations splits.
In particular, this means that we can naturally identify $(V/V')^{\Ga}$ and $V^{\Ga}/(V')^{\Ga}$.
In this paper we will focus on the group $\Ga=\GL_d(K)$. The field $K$ is algebraically closed and of characteristic zero, and $\GL_d(K)$ is thus reductive. 
Assume now that $\Ga$ acts on an affine variety $Y$. Geometric invariant theory deals with studying possible quotients of $Y$ by the action of $\Ga$ by considering the corresponding action of $\Ga$ on the coordinate ring $K[Y]$.  
We summarize in the following theorem well known results from geometric invariant theory which we will use in this paper:
\begin{theorem}
Let $Y$ be an affine variety, and let $\Ga$ be a reductive group acting on $Y$.
\begin{enumerate}
\item If $W_1$ and $W_2$ are two closed $\Ga$-stable subsets of $Y$ then there is $f\in K[Y]^{\Ga}$ such that $f(W_1)=0$ and $f(W_2)=1$  
\item The ring of invariants $K[Y]^{\Ga}$ is finitely generated and therefore its maximal spectrum $Z=Spec_m(K[Y]^{\Ga})$ is an affine variety. 
\item The map $\pi:Y\to Z$ induced from the inclusion $K[Y]^{\Ga}\hookrightarrow K[Y]$ is surjective. 
\item The points in $Z$ are in one to one correspondence with the closed $\Ga$-orbits in $Y$. 
\end{enumerate}
\end{theorem}
\begin{proof}
The first statement is Lemma 3.3. in \cite{Newstead}. 
The second statement is the celebrated theorem of Nagata, see Theorem 3.4. in \cite{Newstead}. The third and fourth statement follow from Theorem 3.5. in \cite{Newstead}. The correspondence between closed orbits in $Y$ and points in $Z$ is induced by $\Ga\cdot y\mapsto \pi(\Ga\cdot y)$.
\end{proof}

\subsection{Linear algebra and natural identifications}\label{subsec:natid}
As before, assume that $K$ is an algebraically closed field of characteristic zero. Let $W$ be a finite dimensional $K$-vector space. 
We write $$W^{p,q}:=W^{\ot p}\ot (W^*)^{\ot q}.$$
This space is naturally identified with $\Hom_K(W^{\ot q,},W^{\ot p})$, where $w_1\ot\cdots w_p\ot g_1\ot\cdots \ot g_q$ corresponds to the linear transformation $$w'_1\ot'\cdots\ot w'_q\mapsto g_1(w'_1)\cdots g_q(w'_q)w_1\ot\cdots\ot w_p.$$
By using the natural identification $W^{**}\cong W$ we see that $(W^{p,q})^*\cong W^{q,p}$. The pairing $W^{p,q}\ot W^{q,p}\to K$ is given by 
$$
\big(w_1\ot\cdots\ot w_p\ot g_1\ot\cdots\ot g_q\big)\ot\big(w'_1\ot\cdots\ot w'_q\ot g'_1\ot\cdots\ot g'_p\big)\mapsto $$\begin{equation}\label{eq:duality}g_1(w'_1)\cdots g_q(w'_q)g'_1(w_1)\cdots g'_p(w_p).\end{equation}
Using the identification $W^{p,q}\cong \Hom_K(W^{\ot q},W^{\ot p})$ this is the same as the map
\begin{equation}\label{eq:tracepairing}\Hom_K(W^{\ot q},W^{\ot p})\ot \Hom_K(W^{\ot p},W^{\ot q})\to K\end{equation}
$$T_1\ot T_2\mapsto \Tr(T_1T_2)$$
We will write $ev_{p,q}:W^{p,q}\to W^{p-1,q-1}$ for the map
$$ev_{p,q}(w_1\ot\cdots\ot w_p\ot g_1\ot\cdots\ot g_q) = g_q(w_p)w_1\ot\cdots\ot w_{p-1}\ot g_1\ot\cdots\ot g_{q-1}.$$ We will omit the subscripts when no confusion can arise. 
We write $$ev^j:=ev_{p-j+1,q-j+1}ev_{p-j+2,q-j+2}\cdots ev_{p,q}.$$ Notice that in this setting $ev^n:W^{n,n}\to W^{0,0} = K$ is the map $T\mapsto \Tr(T)$.

If $U$ is any finite dimensional vector space then $T(U)$ denotes the tensor algebra on $U$ and $K[U]$ the algebra of polynomial functions of $U$. Both algebras are graded, and we have a natural identification
$$Z:(T(U^*)_n)_{S_n}\cong K[U]_n$$
$$Z(\ol{g_1\ot\cdots\ot g_n})(u) = g_1(u)\cdots g_n(u).$$
Where the action of $S_n$ on $T(U^*)_n = (U^*)^{\ot n}$ is given by permuting the tensor factors, and where for a group $G$ and a $G$-representation $V$, we write $V_G$ for the $G$-coinvariants $$V_G:= V/span\{g\cdot v-v\}_{g\in G,v\in V}.$$

The following lemma will be useful for comparing the invariants and coinvariants:
\begin{lemma}\label{lem:coinvariants}
\begin{enumerate}
\item Let $G$ be a reductive group acting on a vector space $V$. The composition of the natural inclusion $V^G\to V$ and the natural surjection $V\to V_G$ gives an isomorphism from the $G$-invariants to the $G$-coinvariants of $V$.
\item Let $G$ and $V$ be as in the previous part, and let $H$ be another reductive group acting on $V$ in such a way that $h(gv) = g(hv)$ for $h\in H, g\in G$ and $v\in V$. Then the group $G\times H$ acts on $V$ and we have natural isomorphisms:
$$(V^G)_H\cong V^{G\times H}\cong (V_H)^G.$$
\end{enumerate}
\end{lemma}
\begin{proof}
For the first assertion we use the fact that the reductivity of $G$ implies that the inclusion $V^G\to V$ splits. We can thus write $V=V^G\oplus V'$ for some $G$-representation $V'$. It holds that $V'^G=0$ and thus $V'_G=0$, since otherwise we have a surjection $V'\to V'_G$ which must split, contradicting the fact that $V'^G\neq 0$. Taking $G$-coinvariants of $V$ we get 
$$V_G = (V^G)_G\oplus V'_G= V^G\oplus 0 = V^G.$$ The map $V^G\to V\to V_G$ is then an isomorphism. 

For the second assertion, we notice that under the assumptions we made $H$ acts on $V^G$ and $G$ acts on $V_H$ in a natural way. Moreover, the isomorphism $V_H\to V\to V^H$ is an isomorphism of $G$-representations. Thus,
$$(V_H)^G\cong (V^H)^G\cong V^{H\times G}\cong (V^G)^H\cong (V^G)_H.$$
\end{proof}

\subsection{Algebraic structures}\label{subsec:algstr}
Recall from the introduction that an algebraic structure of type $((p_i,q_i))$ is given by a finite dimensional vector space $W$ and a collection of structure tensors $x_i\in W^{p_i,q_i}$. In other words- a point $(x_i)$ in the vector space $U(W) = \oplus_i W^{p_i,q_i}$. We will also refer to $(W,(x_i))$ as an algebraic structure in the sequel. A basis $B=\{e_i\}$ for $W$ gives us a dual basis $\{f_i\}$ for $W^*$. The tensor products of these bases gives us then also bases for $W^{p,q}$ for every $(p,q)\in \N^2$ and we can write
$$x_i = \sum_{j_1,\ldots,j_{p_i},k_1,\ldots,k_{q_i}}a^i_{j_1,\ldots,j_{p_i},k_1,\ldots,k_{q_i}}e_{j_1}\ot\cdots\ot e_{j_{p_i}}\ot f_{k_1}\ot\cdots\ot f_{k_{q_j}}.$$
The scalars $a^i_{j_1,\ldots,j_{p_i},k_1,\ldots,k_{q_i}}$ will be referred to as the \emph{structure constants} of $x_i$ with respect to the basis $B$. We will refer to them just as $a^{\bullet}_{\bullet}$ to ease notations. 

The group $\GL(W)$ acts on $U(W)$ in a natural way. The following result holds:
\begin{lemma} The $\GL(W)$-orbits in $U(W)$ correspond exactly to the isomorphism classes of structures of type $((p_i,q_i))$ on $W$. 
\end{lemma}
\begin{proof}
Indeed, two sets of structure tensors $(x_i)$ and $(y_i)$ define isomorphic structures if and only if there is an invertible linear map $g\in \GL(W)$ that intertwines $x_i$ and $y_i$ for every $i$. In other words, for every $i$ the diagram
$$\xymatrix{
W^{\ot q_i}\ar[r]^{g^{\ot q_i}}\ar[d]^{x_i} & W^{\ot q_i}\ar[d]^{y_i}\\ 
W^{\ot p_i}\ar[r]^{g^{\ot p_i}} & W^{\ot p_i}
}
$$
should commute. 
 But this is the same as saying that $g(x_i) = y_i$ for every $i$. See also Lemma 3.2. in \cite{meir2} for a particular instance of this.
\end{proof}
In most cases we will be interested only in algebraic structures which satisfy certain set of axioms $\T$. Such structures often form a Zariski closed $\GL(W)$-stable subset $Y(W)\subseteq U(W)$, where we think here of $U(W)$ as an affine space. For example, a multiplication $m$ is given by the tuple of scalars $(m_{i,j}^k)_{i,j,k}$ for which $e_ie_j = \sum_k m_{i,j}^ke_k$. Associativity of $m$ is given by the polynomial equations $$\sum_a m_{i,j}^am_{a,k}^b = \sum_a m_{i,a}^bm_{j,k}^a$$ for every $i,j,k,b$. 

We will study in detail the defining ideal of $Y(W)$ in Section \ref{sec:axid}. In the terminology of the introduction we have $Y_d = Y(K^d)$ and $U_d = U(K^d)$.   

\subsection{Schur-Weyl duality}\label{subsec:SW}
Schur-Weyl duality describes the link between invariants with respect to general linear groups and representations of the symmetric groups. We recall here the details. 
For a finite dimensional vector space $V$ and $n\in \N$, $V^{\ot n}$ is a representation of $S_n$ in a natural way. A permutation $\sigma\in S_n$ acts via the formula
\begin{equation}\sigma \cdot (v_1\ot\cdots\ot v_n) = v_{\sigma^{-1}(1)}\ot\cdots\ot v_{\sigma^{-1}(n)}.\end{equation}
We denote the resulting linear map $V^{\ot n}\to V^{\ot n}$ by $L^{(n)}_{\sigma}\in \End(V^{\ot n})$. This map commutes with the natural diagonal action of $\GL(V)$. 
To state Schur-Weyl duality, we first recall that a \emph{partition} $\la$ of a natural number $n$ is a non-decreasing sequence $n_1\geq n_2\geq\cdots\geq n_r>0$ such that $\sum_i n_i = n$. We write  $\la\parti n$. Every partition can be thought of geometrically as a \emph{Young diagram}. There is one-to-one correspondence between partitions of $n$ and irreducible representations of $S_n$, which we shall write here as $\la\leftrightarrow \S_{\la}$. Following \cite{Sagan} we call $\S_{\la}$ the \emph{Specht module} corresponding to $\la$. We will write $r(\la)$ for the number of non-zero rows in the partition $\la$.
Schur-Weyl duality is the following statement (see also the discussion in I.1 and Theorem 4.3. in \cite{Procesi}):
\begin{theorem}[Schur-Weyl duality]\label{thm:SW}
\begin{enumerate} 
\item The linear map 
$$\Phi_V:KS_n\to (\End(V)^{\ot n})^{\GL(V)}$$ $$\sigma\mapsto L^{(n)}_{\sigma}$$ is a surjective ring homomorphism.
\item Consider the Wedderburn decomposition of the group algebra of $S_n$, $$KS_n \cong \bigoplus_{\la\parti n} \End(\S_{\la}).$$ The kernel of $\Phi_V$ is $$\bigoplus_{\substack{\la\parti n \\ r(\la)> \dim(V)}}\End(\S_{\la})$$
where $r(\la)$ is the number of non-zero rows in $\la$. 
As a result, we have an isomorphism of algebras $$(\End(V)^{\ot m})^{\GL(V)}\cong \bigoplus_{\substack{\la\parti n \\ r(\la)\leq \dim(V)}} \End(\S_{\la}).$$ 
\end{enumerate}
\end{theorem}
\begin{remark} 
Another way to describe the kernel of $\Phi_V$ is the following: 
If $\dim(V)\geq n$ then $\Phi_V$ is injective, and if $\dim(V)<n$ then $\Ker(\Phi_V)$ is the two-sided ideal of $KS_n$ generated by the idempotent \begin{equation}
\frac{1}{(d+1)!}\sum_{\sigma\in S_{d+1}}(-1)^{\sigma}\sigma\end{equation}
where $d=\dim(V)$ (See Proposition 4.1 in \cite{meir2}). 
\end{remark}

\subsection{Block decomposition permutations}\label{subsec:permutations}For two integers $a\leq b$ we use the notation $$[a,b] = \{n\in \N| a\leq n\leq b\}.$$
Since we will not use the real numbers in this paper this will cause no confusion with intervals in the real line.

An \emph{unordered partition } $\la=(n_1,\ldots n_k)$ of $n\in \N$ is a sequence of integers which satisfies $\sum_i n_i = n$ (for regular partitions we also require that $n_1\geq n_2\geq\cdots\geq n_k$).
We write $$S_{\la} = S_{n_1,\ldots,n_r}:= S_{n_1}\times\cdots\times S_{n_k}$$
and $$I^{\la}_i = [(\sum_{j=1}^{i-1}n_j)  + 1, \sum_{j=1}^i n_j]$$ 
The unordered partition $\la$ gives rise to two maps $\Pi_{\la}:S_{\la}\to S_n$ and $\Omega^{\la}:S_k\to S_n$ which will be used in this paper.
\begin{definition}\label{def:Pi} We define 
$$\Pi_{\la}: S_{\la}\to S_n$$ to be the group embedding for which $\Pi(S_{n_i})$ permutes the elements of $I^{\la}_i$. 
\end{definition}
In places where it will cause no confusion we will identify between $(\sigma_i)\in S_{\la}$ and its image under $\Pi_{\la}$ in $S_n$. 

Next we give the definition of block permutations. 
\begin{definition}\label{def:Omega}
The permutation $\Omega^{\la}(\sigma)\in S_n$ is the unique permutation which satisfies the following conditions:
\begin{enumerate}
\item If $x\in I^{\la}_i$, $y\in I^{\la}_j$ and $i\neq j$ then $\Omega^{\la}(\sigma)(x)<\Omega^{\la}(\sigma)(y)$ if and only if $\sigma(i)<\sigma(j)$.
\item If $x,y\in I^{\la}_i$ then $\Omega^{\la}(\sigma)(x)<\Omega^{\la}(\sigma)(y)$ if and only if $x<y$
\end{enumerate}
\end{definition}
In other words, $\Omega^{\la}(\sigma)$ rearranges the numbers $\{1,\ldots,n\}$ in such an order that the elements of $I^{\la}_{\sigma^{-1}(1)}$ appear first, the elements of $I^{\la}_{\sigma^{-1}(2)}$ appear after them and so on. It does not change the inner order of the elements of each of the $I^{\la}_i$'s. 
If, for example $k=3$, $n=5$ and $\la = (1,2,2)$ then 
$$\Omega^{\la}((13)) =   \Bigl(\begin{matrix}
    1 & 2 & 3 & 4 & 5 \\
    4 & 5 & 2 & 3 & 1 
  \end{matrix}\Bigr).
$$
In cycles notation we get $\Omega^{\la}((13)) = (14325)$ which already shows that $\Omega^{\la}$ is in general not a group homomorphism, as it sends an element of order 2 to an element of order 5. 
In fact $\Omega^{\la}$ is a group homomorphism if and only if $n_1=n_2=\ldots =n_k$. In this special case $\Omega^{\la}$ can be written explicitly. Write first $m=n_1=\cdots=n_k$ and $$\beta:[1,k]\times [1,m]\to [1,n]$$ $$(i,j)\mapsto (i-1)m + j.$$
The map $\beta$ is a bijection and it induces a natural group homomorphism $\ol{\beta}:S_k\times S_m\to S_n$ given by $$\ol{\beta}(\sigma,\tau)(x) = \beta(\sigma(i),\tau(j))$$ where $\beta^{-1}(x)= (i,j).$ The map $\Omega^{\la}$ is then given by the composition $$S_k\to S_k\times S_m\stackrel{\ol{\beta}}{\to} S_n$$ where the first map is given by $\sigma\mapsto (\sigma,\Id)$.

\begin{remark}\label{rem:uptoconj}
Up to conjugation in $S_n$ this map is also given by the composition 
$$S_k\to (S_k)^m\stackrel{\Pi_{(k,k,\ldots,k)}}{\longrightarrow} S_n$$ where the first map is the diagonal embedding. This will be used in Section \ref{sec:hilbert} to derive formulas for the Hilbert function of the universal ring of invariants.
\end{remark}


\begin{definition}\label{def:alpha}
Assume that $n=\sum_{i=1}^k n_ip_i$. We write $$\alpha^{(p_i)}_{(n_i)}:S_{(n_i)}\to S_n$$ for the following group homomorphism
$$S_{(n_i)} = S_{n_1}\times\cdots\times S_{n_k}\stackrel{\Omega^{\la_1}\times\cdots\times \Omega^{\la_k}}{\longrightarrow } 
S_{n_1p_1}\times\cdots\times S_{n_kp_k}\stackrel{\Pi_{(n_ip_i)}}{\rightarrow }S_n$$
where $\la_i = (\underbrace{p_i,p_i,\ldots,p_i}_{n_i \text{ times}})$.
\end{definition}

\subsection{Littlewood-Richardson and Kronecker coefficients}
When $n=a+b$ we have an embedding $\Pi_{(a,b)}:S_a\times S_b\to S_n$. 
We can restrict representations of $S_n$ along $\Pi_{(a,b)}$ and write 
$$\Pi^*_{(a,b)}([\S_{\la}]) = \sum_{\la_a\parti a, \la_b\parti b}c^{\la}_{\la_a,\la_b}[\S_{\la_a}\boxtimes \S_{\la_b}]= \sum_{\la_a\parti a, \la_b\parti b}c^{\la}_{\la_a,\la_b}[\S_{\la_a}]\ot [\S_{\la_b}],$$ where $\S_{\la_a}\boxtimes\S_{\la_b}$ stands for the vector space $\S_{\la_a}\ot \S_{\la_b}$ with the tensor product action of $S_a\times S_b$. The term $[V]$ stands for the class of the representation $V$ in the relevant Grothendieck group, where we use the identification $\text{K}_0(G\times H)\cong \text{K}_0(G)\ot \text{K}_0(H)$. 
The coefficients $c^{\la}_{\la_a,\la_b}$ are called the \emph{Littlewood-Richardson coefficients}. Following \cite{meir5} we will give the following generalization when the unordered partition which appears in $\Pi$ has more than two components. 
\begin{definition}\label{def:LRco}
Let $n=n_1+n_2+\ldots +n_k$, and let $\Pi_{(n_1,\ldots,n_k)}:S_{n_1}\times\cdots\times S_{n_k}\to S_n$ be the group homomorphism defined in \ref{subsec:permutations}.
We write 
$$\Pi_{(n_1,\ldots,n_k)}^*([\S_{\la}]) = \sum_{\substack{\la_1\parti n_1\\ \vdots\\ \la_k\parti n_k}}c^{\la}_{(\la_1,\ldots,\la_k)}[\S_{\la_1}]\boxtimes\cdots\boxtimes [\S_{\la_k}]$$
The coefficients $c^{\la}_{(\la_1,\ldots,\la_k)} = c^{\la}_{(\la_i)}$ are called the iterated Littlewood-Richardson coefficients.
\end{definition}
For the Kronecker coefficients, let $\la$ and $\mu$ be two partitions of $n$. The tensor product of the Specht modules $\S_{\la}\ot\S_{\mu}$ with the diagonal $S_n$-action splits as a direct sum of simple $S_n$-representations and we can write 
$$\S_{\la}\ot \S_{\mu} \cong \bigoplus_{\nu\parti n } \S_{\nu}^{g(\la,\mu,\nu)}.$$
The natural numbers $g(\la,\mu,\nu)$ are known as the \emph{Kronecker cofficients}. For more on Kronecker coefficients and Littlewood-Richardson coefficients see \cite{BVO}.
Since all the representations of $S_n$ are self dual, the Kronecker coefficients are also given by the formula  
$$g(\la,\mu,\nu) = \dim\Hom_{S_n}(\S_{\la}\ot\S_{\mu}\ot \S_{\nu},\one).$$
If $\la_1,\ldots, \la_k$ are partitions of $n$ we define the \emph{iterated Kronecker coefficients } to be $$g(\la_1,\la_2,\ldots,\la_k):=\dim\Hom_{S_n}(\S_{\la_1}\ot\cdots\ot \S_{\la_k},\one).$$
Another way to describe the iterated Kronecker coefficients is the following:
if we write $diag:S_n\to (S_n)^k$ for the diagonal embedding, then
\begin{equation}\label{eq:Kronecker} diag^*([\S_{\la_1}\boxtimes\cdots\boxtimes \S_{\la_k}]) = \sum_{\la\parti n}g(\la_1,\ldots,\la_k,\la)[\S_{\la}]\end{equation}
In the relevant Grothendieck rings. 
We will use this identification in Section \ref{sec:hilbert} in calculating the Hilbert functions of $\Ainf$.

\subsection{Commutative bialgebras and Hopf algebras}
The $K$-algebras which we will consider in this paper are commutative and graded by the monoid $\N^r$. 
Recall that such an algebra $A$ is said to be a \emph{bialgebra} if it is equipped with maps $\epsilon:A\to K$ and $\Delta:A\to A\ot A$ 
such that $(A,\Delta,\epsilon)$ is a coalgebra, which means that the dual axioms to that of a unital algebra are satisfied, and such that $\epsilon$ and $\Delta$ are algebra maps. This bialgebra is said to be \emph{graded} if $\Delta$ and $\epsilon$ preserve the grading, where it is understood that $K$ has degree $(0,0,\ldots,0)$ and $A_{n_1,\ldots,n_r}\ot A_{m_1,\ldots,m_r}$ has degree $(n_1+m_1,\ldots,n_r+m_r)$. This means in particular that $\epsilon(A_{n_1,\ldots n_r})=0$ unless $(n_1,\ldots,n_r) = (0,0,\ldots,0)$.
A bialgebra is said to be a \emph{Hopf algebra} if it admits an antipode. This is a linear map $S:A\to A$ which satisfies
$$\forall a\in A:\quad m(S\ot 1)\Delta(a) = m(1\ot S)\Delta(a) = \epsilon(a)1.$$ The antipode $S$, if it exists, can be understood as the inverse of $\Id_A\in Hom_K(A,A)$ under the convolution product induced by $\Delta$ and $m$. 


A \emph{group-like element} in a bialgebra $A$ is an element $g\in A$ which satisfies $\Delta(g)=g\ot g$ and $\epsilon(g)=1$.
The set of group-like elements in $A$ forms a monoid which we denote by $G(A)$. In case $A$ is a Hopf algebra this monoid is in fact a group, where the inverse of $g\in G(A)$ is $S(g)$. 
 
A \emph{primitive element} in a bialgbera $A$ is an element $x\in A$ which satisfies the equations $\Delta(x) = x\ot 1 + 1\ot x$ and $\epsilon(x)=0$.  
The set of primitive elements in $A$ is a $K$-subspace which we will denote by $\Pp(A)$. 
All the bialgebras we will encounter in this paper will be either generated by group-like elements or by primitive elements. 
For more on Hopf algebras see \cite{Sweedler}, in particular Chapters III and IV. 

\section{A proof of Theorem \ref{thm:mainstructure}}\label{sec:proofthm1}
\subsection{The non-invariant coordinate algebra}\label{subsec:bigalgebra}
Let $W=K^d$ and let $$U(W)=\bigoplus_{i=1}^r W^{p_i,q_i}.$$
We will write $U=U(W)$ when $W$ is fixed. 

We write an element in $U$ as $(x_i)_{i=1}^r$. We thus think of points of $U$ as possible algebraic structures of type $((p_i,q_i))$ on $W$. 
Write $U_{(i)}=W^{p_i,q_i}$. Then $U=\bigoplus_i U_{(i)}$
and we have a natural isomorphism 
$$K[U] = K[U_{(1)}\oplus U_{(2)}\oplus\cdots\oplus U_{(r)}]\cong
K[U_{(1)}]\ot\cdots\ot K[U_{(r)}]$$
For every $i$ the degree of polynomials gives a grading on $K[U_{(i)}]$ by $\N$. This gives  a grading on $K[U]$ by $\N^r$. Using the isomorphism $Z$ of Subsection \ref{subsec:natid} we have 
$$K[U]_{n_1,\ldots,n_r} = K[U_{(1)}]_{n_1}\ot\cdots\ot K[U_{(r)}]_{n_r}\cong $$
$$(U_{(1)}^*)^{\ot n_1}_{S_{n_1}}\ot\cdots\ot (U_{(r)}^*)^{\ot n_r}_{S_{n_r}}\cong$$
$$((U_{(1)}^*)^{\ot n_1}\ot\cdots\ot (U_{(r)}^*)^{\ot n_r})_{S_{n_1}\times\cdots\times S_{n_r}}.$$

Write $$U_{n_1,\ldots, n_r}:=((U_{(1)}^*)^{\ot n_1}\ot\cdots\ot (U_{(r)}^*)^{\ot n_r}).$$ 
This vector space is isomorphic to 
$W^{n',n}$ where $n'=\sum_i q_in_i$ and $n=\sum_i n_ip_i$. 

We begin by studying the action of $S_{n_1,\ldots,n_r}=S_{n_1}\times\cdots\times S_{n_r}$ on this vector space. 
For this, it will be enough to study the action of $S_n$ on $(V^*)^{\ot n}$ where $V=W^{p,q}$ .
The action of $\sigma\in S_n$ is given by $$\sigma\cdot (t_1\ot t_2\ot\cdots\ot t_n) = t_{\sigma^{-1}(1)}\ot\cdots\ot t_{\sigma^{-1}(n)}$$
where $t_i\in T^*\cong W^{q,p}$ (see Equation \ref{eq:duality}). If every $t_i$ is a basic tensor of the form $t_i= w_{i1}\ot\cdots\ot w_{iq}\ot f_{i1}\ot\cdots\ot f_{ip}$ then we see that after applying the natural isomorphism $(T^*)^{\ot n}\cong W^{nq,np}$ given by grouping all the $W$ tensorands before all the $W^*$ tensoarnds, the permutation $\sigma\in S_n$ acts on $(T^*)^{\ot n}$ by the formula:
$$\sigma\cdot (w_{11}\ot\cdots\ot w_{1q}\ot \cdots\ot w_{n1}\ot\cdots\ot w_{nq}\ot f_{11}\ot\cdots\ot f_{1p}\ot\cdots\ot f_{n1}\ot\cdots\ot f_{np})= $$
$$w_{\sigma^{-1}(1)1}\ot\cdots\ot w_{\sigma^{-1}(1)q}\ot f_{\sigma^{-1}(1)1}\ot\cdots\ot f_{\sigma^{-1}(1)p}\ot \cdots\ot f_{\sigma^{-1}(n)1}\ot\cdots\ot f_{\sigma^{-1}(n)p}.$$

Identifying $W^{nq,np}$ with $\Hom_K(W^{\ot np},W^{\ot nq})$, this is the same as
$$\sigma\cdot T = L^{(nq)}_{\Omega^{(p^n)}(\sigma)}T(L^{(np)}_{\Omega^{(p^n)}(\sigma^{-1})})$$
where we use the terminology of Definition \ref{def:Omega}.

Going now back to the general case, the permutations $\Omega^{(p^n)}(\sigma)$ glue together to give the permutation $\alpha^{(p_i)}_{(n_i)}$, and similarly when we replace $(p_i)$ by $(q_i)$, see Definition \ref{def:alpha}. The conclusion of this is the following: write $n = \sum_i p_in_i$ and $n'=\sum_i q_in_i$. 
\begin{lemma}\label{lem:actionformula}
The action of $S_{n_1,\ldots n_r}$ on $U_{n_1,\ldots, n_r}$ is given by the formula:
$$(\sigma_i)\cdot T = L^{(n')}_{\alpha^{(q_i)}_{(n_i)}(\sigma_i)}T(L^{(n)}_{\alpha^{(p_i)}_{(n_i)}(\sigma_i^{-1})})$$
\end{lemma}

\subsection{The action of $\Ga$ on $K[U]$ and the algebra of invariants}\label{subsec:invariantalg}
Since the action of $S_{n_1}\times\cdots\times S_{n_r}$ commutes with the action of $\Ga=\GL(W)$ and since finite groups are reductive in characteristic zero, Lemma \ref{lem:coinvariants} gives 
$$(K[U]_{n_1,\ldots n_r})^{\Ga}\cong( (U_{n_1,\ldots, n_r})_{S_{n_1}\times\cdots\times S_{n_r}})^{\Ga}\cong (U_{n_1,\ldots, n_r}^{\Ga})_{S_{n_1}\times\cdots\times S_{n_r}}.$$
We have the following isomorphism of $\Ga$-representations:
$$U_{n_1,\ldots ,n_r}\cong W^{n',n}$$ where $n'=\sum_i n_iq_i$ and $n=\sum_i n_ip_i$. By Theorem \ref{thm:SW} we get the following description of $U_{n_1,\ldots n_r}^{\Ga}$:
\begin{enumerate}
\item If $n\neq n'$ then $(U_{n_1,\ldots, n_r})^{\Ga}=0$.
\item If $n=n'$ then $(U_{n_1,\ldots, n_r})^{\Ga}$ is spanned by $(L^{(n)}_{\sigma})_{\sigma\in S_n}$ where we identify $U_{n_1,\ldots, n_r}\cong W^{n,n}\cong \End_K(W^{\ot n})$. Moreover, the linear relations between the elements $L^{(n)}_{\sigma}$ are the following:
If $n\leq d$ then $L^{(n)}_{\sigma}$ are all linearly independent, and if $n>d$ then the linear relations are spanned by relations of the form 
$$\sum_{\sigma\in S_{d+1}}(-1)^{\sigma}L^{(n)}_{\tau_1\sigma\tau_2}$$ where $\tau_1,\tau_2\in S_n$.
\end{enumerate}
By Lemma \ref{lem:actionformula} the action of $(\sigma_i)\in S_{n_1,\ldots,n_r}$ on $(U_{n_1,\ldots, n_r})^{\Ga}$ is given by 
$$(\sigma_i)\cdot L^{(n)}_{\sigma} = L^{(n)}_{\alpha^{(q_i)}_{(n_i)}(\sigma_i)\sigma\alpha^{(p_i)}_{(n_i)}(\sigma_i^{-1})}$$

We can now describe explicitly the space $(K[U]_{n_1,\ldots,n_r})^{\Ga}$: 
By following the duality isomorphisms and using Equation \ref{eq:tracepairing} we see that the image of the element $L^{(n)}_{\sigma}$ in $(U_{n_1,\ldots n_r}^{\Ga})_{S_{n_1,\ldots n_r}}$ is the polynomial map \begin{equation}p(n,\sigma,n_1,\ldots,n_r)((x_i)):= Tr_{W^{\ot n}}\big(L^{(n)}_{\sigma}\circ(\xini)\big).\end{equation}
In the sequel we will evaluate these polynomials for algebraic structures of different dimensions. When it will be necessary to indicate the specific vector space we will also write $p(n,\sigma,n_1,\ldots,n_r)(W,(x_i))$. 
This discussion can be summarized in the following proposition:
\begin{proposition}\label{prop:invpolys}
Write $n=\sum_i p_in_i$ and $n'=\sum_i q_in_i$. If $n\neq n'$ then $K[U]^{\Ga}_{n_1,\ldots n_r}=0$. If $n=n'$ then $K[U]^{\Ga}_{n_1,\ldots,n_r}$
is spanned by the polynomials $p(n,\sigma,n_1,\ldots,n_r)$ for $\sigma\in S_n$. The linear relations between these polynomials are spanned by the following two types of relations:
\begin{enumerate}
\item[R1] For every $(\sigma_i)\in S_{n_1,\ldots,n_r}$ we have \begin{equation}\label{eq:cyclicity}p(n,\sigma,n_1,\ldots,n_r) = p(n,\alpha^{(q_i)}_{(n_i)}(\sigma_i)\sigma\alpha^{(p_i)}_{(n_i)}(\sigma_i^{-1}),n_1,\ldots,n_r).\end{equation}
\item[R2] If $n>d$ then for every $\tau_1,\tau_2\in S_n$ we have $$\sum_{\sigma\in S_{d+1}}(-1)^{\sigma}p(n,\tau_1\sigma\tau_2,n_1,\ldots ,n_r)=0.$$

\end{enumerate}
\end{proposition}

\begin{definition}
We call the polynomials $p(n,\sigma,n_1,\ldots, n_r)$ the basic polynomial invariants of multi-degree $(n_1,\ldots,n_r)$.
\end{definition}
\begin{proof}[Proof of Theorem \ref{thm:mainstructure}]
Proposition \ref{prop:invpolys} gives a linear description of the ring of invariants $K[U_d]^{GL_d(K)}$. The product of two basic invariants is given in Proposition \ref{prop:ainfform} and follows from the fact that if $T_i:W_i\to W_i$ for $i=1,2$ then 
$$\Tr_{W_1\ot W_2}(T_1\ot T_2) = \Tr_{W_1}(T_1)\Tr_{W_2}(T_2).$$
For relations of type $R3$, we write $I_{\T,d}:= \Ker(K[U_d]\to K[Y_d])$. We thus have a short exact sequence 
$$0\to I_{\T,d}\to K[U_d]\to K[Y_d]\to 0$$ which gives the short exact sequence 
$$0\to I_{\T,d}^{\GL_d(K)}\to K[U_d]^{\GL_d(K)}\to K[Y_d]^{\GL_d(K)}\to 0.$$
This finishes the proof of Theorem \ref{thm:mainstructure}. We Will give in Section \ref{sec:axid} a uniform description of $I_{\T,d}^{\GL_d(K)}$. 
\end{proof}

Recall that $$X_d = \{[(W,(x_i))]| (W,(x_i))\text{ is a $d$-dimensional structure with a closed } \GL_d(K)\text{ orbit} \}\text{ and }$$
$$X = \bigsqcup_{d\geq 0} X_d.$$
A closed $\GL_d(K)$ orbit means here that $\GL_d(K)\cdot ((x_i))$ is Zariski-closed inside $U(K_d)$. 
For every $d$, $X_d$ has the structure of an affine variety, and $K[X_d] = K[U(K^d)]^{\GL_d}$ where $U(K^d)$ is described in the beginning of Subsection \ref{subsec:bigalgebra}.
The first definition of $K[X]$ (Definition \ref{def:firstdefkx}) gives us a natural surjective homomorphism 
$$\Phi_d:K[X]\to K[U_d]^{\GL_d(K)}=K[X_d].$$ 
\begin{definition} We write $I_d$ for the kernel of $\Phi_d$.  
It is spanned by elements of the form 
$$\sum_{\sigma\in S_{d+1}}(-1)^{\sigma}p(n,\tau_1\sigma\tau_2,n_1,\ldots,n_r)$$ where $\tau_1,\tau_2\in S_n$ and $n= \sum_i p_in_i = \sum_i q_in_i$ is bigger than $d$. 
\end{definition}

Since $I_d \cap(\Ainf)_{n_1,\ldots,n_r}$ is non-zero if and only if $\sum_i n_ip_i = \sum_i n_iq_i =n >d$ and all the ideals $I_d$ are graded, it holds that 
$$\bigcap_{d\geq 0} I_d = 0$$. 

Write $K^X$ for the commutative algebra of all functions from $X$ to $K$ with pointwise addition and multiplication. The maps $\{\Phi_d\}$ for $d\geq 0$ glue together to give a map $\Xi: \Ainf\to K^X$ with kernel $\bigcap I_d=0$. The map $\Xi$ is thus injective and enables us to think of elements of $K[X]$ as functions from $X$ to $K$. 
One might ask if the image of $\Xi$ separates points in $X$. That is- if for every two elements $x_1\neq x_2$ of $X$ there is an element $c\in \Ainf$ such that $\Xi(c)(x_1)\neq \Xi(c)(x_2)$. This is almost the case. From invariant theory (see Subsection \ref{subsec:GIT}) we know that elements of $\Xi(\Ainf)$ can separate two different points in $X_d$ for a given $d$. However, the image might fail to separate $x_1\in X_{d_1}$ and $x_2\in X_{d_2}$ for $d_1\neq d_2$. 

Consider for example the case of an algebraic structures which is a vector space with a single linear endomorphism. More specifically, consider the two structures 
$(K,T_1)$ and $(K^2,T_2)$ where $$T_1=\Id:K\to K$$ and $$T_2= \begin{pmatrix} 1 & 0 \\ 0 & 0 \end{pmatrix}:K^2\to K^2.$$ In this case all the invariants will be generated by traces of powers of the linear endomorphism (see Section \ref{sec:example}), and the fact that $\Tr(T_1^n) = \Tr(T_2^n)=1$ for every $n>0$ means that $\Xi(\Ainf)$ cannot distinguish these two structures.

To overcome this, we will introduce in Section \ref{sec:ainfwinf} another commutative algebra $\winf$ and we will show that $\winf\cong \Ainf\ot K[D]$. We The map $\Xi$ can be extended to $\winf$ by sending $D$ to the function $$[(W,(x_i))]\mapsto \dim_K W.$$ This extension of $\Xi$ to  is still injective, and it has the advantage that it also separates points in $X$.
\begin{definition}\label{def:charinv}
If $(W,(x_i))$ is any algebraic structure then we will denote by $\chi_{(W,(x_i))}:\winf\to K$ the ring homomorphism arising from evaluating closed diagrams on $W$.
We call $\chi_{(W,(x_i))}$ the character of invariants of $W$.
\end{definition}

\section{String diagrams}\label{sec:diagrams}
Let $(W,(x_i))$ be an algebraic structure of type $((p_i,q_i))$.
Throughout this paper we will visualise some linear maps and scalar invariants using a certain type of string diagrams. To this end a \emph{diagram} contains
\begin{enumerate} 
\item A collection of boxes labeled by an element in $\{x_1,\ldots,x_r,\Id_W\}$. A box with label $x_i$ has $q_i$ input strings at the bottom, and $p_i$ output strings at the top. 
A box labeled with $\Id_W$ has one input string and one output string. 
\item Connections between some of the output strings with the input strings (but no connections between two output strings or two input strings). 
\item A permutation of the free (non-connected) input strings and a permutation of the free output strings.
\end{enumerate}
Each such diagram can be interpreted as a linear transformation. 
A box with label $x_i$ corresponds to the linear transformation given by $x_i$, and a box with label $\Id_W$ corresponds to the identity map of $W$. Placing boxes one next to the other corresponds to taking the tensor product of the corresponding linear transformations. Permuting the strings corresponds to permuting the tensor factors of $W$, and connecting an input string with an output string correspond to applying the evaluation map $ev:W\ot W^*\to K$ on two particular copies of $W$ and $W^*$. 
If a diagram has $q$ free (non-connected) input strings and $p$ free output strings then it represents a linear map in $W^{p,q}\cong \Hom_K(W^{\ot q},W^{\ot p})$.
Such a diagram will be referred to as a diagram of degree $(p,q)$. 
Diagrams of degree $(0,0)$ will be referred to as \emph{closed} diagrams. They represent linear maps $W^{\ot 0} = K\to W^{\ot 0}=K$, or scalars. 
For example, the diagram
\begin{equation}
\begin{tikzpicture}[baseline={(current  bounding  box.center)}]
	\begin{pgfonlayer}{nodelayer}
		\node [style=none] (3) at (-7, 3.5) {};
		\node [style=none] (4) at (-7.25, 2.5) {};
		\node [style=none] (5) at (-6.75, 2.5) {};
		\node [style=none] (6) at (-4.5, 2.5) {};
		\node [style=none] (7) at (-4.25, 3.5) {};
		\node [style=none] (8) at (-4.75, 3.5) {};
		\node [style=none] (9) at (-2.25, 3.5) {};
		\node [style=none] (10) at (-2.25, 2.5) {};
		\node [style=none] (11) at (-7.25, 1.75) {};
		\node [style=none] (12) at (-6.75, 1.75) {};
		\node [style=none] (13) at (-4.75, 4.25) {};
		\node [style=none] (14) at (-4.25, 4.25) {};
		\node [style=none] (15) at (-4.5, 1.75) {};
		\node [style=none] (16) at (-2.25, 4.25) {};
		\node [style=none] (17) at (-2.25, 1.75) {};
		\node [style=none] (18) at (-7, 4.25) {};
		\node [style=1function] (19) at (-2.25, 3) {$x_3$};
		\node [style=2function] (20) at (-4.5, 3) {$x_2$};
		\node [style=2function] (21) at (-7, 3) {$x_1$};
	\end{pgfonlayer}
	\begin{pgfonlayer}{edgelayer}
		\draw (14.center) to (7.center);
		\draw (13.center) to (8.center);
		\draw (18.center) to (3.center);
		\draw (5.center) to (12.center);
		\draw (11.center) to (4.center);
		\draw (6.center) to (15.center);
		\draw (9.center) to (16.center);
		\draw (10.center) to (17.center);
	\end{pgfonlayer}
\end{tikzpicture}
\end{equation}
represents $x_1\ot x_2\ot x_3\in W^{4,4}$ and so has degree $(4,4)$, while  
\begin{equation}
\begin{tikzpicture}[baseline={(current  bounding  box.center)}]
	\begin{pgfonlayer}{nodelayer}
		\node [style=none] (3) at (-6.75, 3.5) {};
		\node [style=none] (4) at (-7, 2.5) {};
		\node [style=none] (5) at (-6.5, 2.5) {};
		\node [style=none] (6) at (-3.25, 2.5) {};
		\node [style=none] (7) at (-5.25, 3.5) {};
		\node [style=none] (8) at (-4.75, 3.5) {};
		\node [style=none] (9) at (-3.25, 3.5) {};
		\node [style=none] (10) at (-5, 2.5) {};
		\node [style=none] (11) at (-6.5, 1.25) {};
		\node [style=none] (12) at (-7, 1.25) {};
		\node [style=none] (13) at (-5.25, 4.25) {};
		\node [style=none] (14) at (-4.75, 4.25) {};
		\node [style=none] (15) at (-5, 1.25) {};
		\node [style=none] (16) at (-3.25, 4.25) {};
		\node [style=none] (17) at (-3.75, 1.75) {};
		\node [style=none] (18) at (-6.75, 4.25) {};
		\node [style=1function] (19) at (-3.25, 3) {$x_3$};
		\node [style=2function] (20) at (-5, 3) {$x_2$};
		\node [style=2function] (21) at (-6.75, 3) {$x_1$};
		\node [style=none] (22) at (-2, 4.25) {};
		\node [style=none] (23) at (-2, 1.75) {};
	\end{pgfonlayer}
	\begin{pgfonlayer}{edgelayer}
		\draw [in=90, out=-90] (14.center) to (7.center);
		\draw [in=90, out=-90] (13.center) to (8.center);
		\draw (18.center) to (3.center);
		\draw [in=90, out=-90, looseness=1.25] (5.center) to (12.center);
		\draw [in=-90, out=90] (11.center) to (4.center);
		\draw [in=75, out=-105, looseness=0.75] (6.center) to (15.center);
		\draw (9.center) to (16.center);
		\draw [in=105, out=-90] (10.center) to (17.center);
		\draw [bend left=270, looseness=1.50] (22.center) to (16.center);
		\draw (22.center) to (23.center);
		\draw [in=-75, out=-90, looseness=1.50] (23.center) to (17.center);
	\end{pgfonlayer}
\end{tikzpicture}
\end{equation}
represents $ev(L^{(4)}_{(23)}(x_1\ot x_2\ot x_3)L^{(4)}_{(12)(34)})\in W^{3,3}$.
Notice that the order of the free input and free output strings is important. This is the reason we include the permutations as part of the data of the diagram. 

\begin{definition}\label{def:eqDi} Two diagrams $Di_1,Di_2$ of degree $(p,q)$ are equivalent if they describe the same linear transformation in $W^{p,q}$ for every finite dimensional vector space $W$ and every collection $(x_i)$ of structure tensors.
This equivalence relation is generated by the following two relations:
\begin{enumerate}
\item \textbf{Permutations relation}: If $Di_1$ can be obtained from $Di_2$ by permuting the boxes while keeping the free input and output strings fixed. So for example we have the following equivalence of diagrams:
\begin{equation}
\begin{tikzpicture}
	\begin{pgfonlayer}{nodelayer}
		\node [style=1function] (0) at (-7, 1) {$x_1$};
		\node [style=1function] (1) at (-5.75, 1) {$x_2$};
		\node [style=none] (2) at (-7, 1.5) {};
		\node [style=none] (3) at (-7, 2.25) {};
		\node [style=none] (4) at (-5.75, 2.25) {};
		\node [style=none] (5) at (-5.75, 1.5) {};
		\node [style=none] (6) at (-5.75, 0.5) {};
		\node [style=none] (7) at (-5.75, -0.25) {};
		\node [style=none] (8) at (-7, 0.5) {};
		\node [style=none] (9) at (-7, -0.25) {};
		\node [style=1function] (10) at (-3.25, 1) {$x_2$};
		\node [style=1function] (11) at (-2.25, 1) {$x_1$};
		\node [style=none] (12) at (-3.25, 1.5) {};
		\node [style=none] (13) at (-3.25, 1.75) {};
		\node [style=none] (14) at (-2.25, 1.75) {};
		\node [style=none] (15) at (-2.25, 1.5) {};
		\node [style=none] (16) at (-2.25, 0.5) {};
		\node [style=none] (17) at (-2.25, 0.25) {};
		\node [style=none] (18) at (-3.25, 0.5) {};
		\node [style=none] (19) at (-3.25, 0.25) {};
		\node [style=none] (20) at (-4.25, 1) {=};
		\node [style=none] (21) at (-2.25, -0.25) {};
		\node [style=none] (22) at (-3.25, -0.25) {};
		\node [style=none] (23) at (-3.25, 2.25) {};
		\node [style=none] (24) at (-2.25, 2.25) {};
	\end{pgfonlayer}
	\begin{pgfonlayer}{edgelayer}
		\draw (3.center) to (2.center);
		\draw (4.center) to (5.center);
		\draw (8.center) to (9.center);
		\draw (6.center) to (7.center);
		\draw (13.center) to (12.center);
		\draw (14.center) to (15.center);
		\draw (18.center) to (19.center);
		\draw (16.center) to (17.center);
		\draw [in=90, out=-90] (24.center) to (13.center);
		\draw [in=90, out=-90] (23.center) to (14.center);
		\draw [in=90, out=-90, looseness=0.75] (17.center) to (22.center);
		\draw [in=105, out=-105, looseness=0.75] (19.center) to (21.center);
	\end{pgfonlayer}
\end{tikzpicture}
\end{equation}
\item \textbf{Identity reduction}: if $Di_1$ can be obtained from $Di_2$ (or vice versa) by using the fact that composition with $\Id_W$ is the identity morphism. More precisely, we have the following equivalences between diagrams:
\begin{equation}\begin{tikzpicture}
	\begin{pgfonlayer}{nodelayer}
		\node [style=multi function] (0) at (2.75, 0.25) {$Di$};
		\node [style=1function] (1) at (6.5, 0.25) {Id$_W$};
		\node [style=none] (2) at (1.5, -0.25) {};
		\node [style=none] (3) at (2.25, -0.25) {};
		\node [style=none] (4) at (4.5, -0.25) {};
		\node [style=none] (5) at (1.5, 0.5) {};
		\node [style=none] (6) at (2.25, 0.5) {};
		\node [style=none] (7) at (4.5, 0.5) {};
		\node [style=none] (8) at (1.5, 1.25) {};
		\node [style=none] (9) at (2.25, 1.25) {};
		\node [style=none] (10) at (4.5, 1.25) {};
		\node [style=none] (11) at (1.5, -0.75) {};
		\node [style=none] (12) at (2.25, -0.75) {};
		\node [style=none] (13) at (4.5, -0.5) {};
		\node [style=none] (14) at (6.5, 0.5) {};
		\node [style=none] (15) at (6.5, 1) {};
		\node [style=none] (16) at (6.5, -0.25) {};
		\node [style=none] (17) at (6.5, -0.5) {};
		\node [style=none] (18) at (5.5, 1) {};
		\node [style=none] (19) at (5.5, -0.5) {};
		\node [style=none] (20) at (3, -0.75) {$\cdots$};
		\node [style=none] (21) at (3, 1.25) {$\cdots$};
		\node [style=none] (44) at (7.5, 0) {$\Large{=}$};
		\node [style=multi function] (68) at (10.75, 0.25) {$Di$};
		\node [style=none] (70) at (9.5, -0.25) {};
		\node [style=none] (71) at (10.25, -0.25) {};
		\node [style=none] (72) at (12.5, -0.25) {};
		\node [style=none] (73) at (9.5, 0.5) {};
		\node [style=none] (74) at (10.25, 0.5) {};
		\node [style=none] (75) at (12.5, 0.5) {};
		\node [style=none] (76) at (9.5, 1.25) {};
		\node [style=none] (77) at (10.25, 1.25) {};
		\node [style=none] (78) at (12.5, 1.25) {};
		\node [style=none] (79) at (9.5, -0.75) {};
		\node [style=none] (80) at (10.25, -0.75) {};
		\node [style=none] (81) at (12.5, -0.75) {};
		\node [style=none] (88) at (11, -0.75) {$\cdots$};
		\node [style=none] (89) at (11, 1.25) {$\cdots$};
		\node [style=multi function] (90) at (6.5, -3) {$Di$};
		\node [style=1function] (91) at (10.25, -3) {Id$_W$};
		\node [style=none] (92) at (5.25, -3.5) {};
		\node [style=none] (93) at (6, -3.5) {};
		\node [style=none] (94) at (8.25, -3.5) {};
		\node [style=none] (95) at (5.25, -2.75) {};
		\node [style=none] (96) at (6, -2.75) {};
		\node [style=none] (97) at (8.25, -2.75) {};
		\node [style=none] (98) at (5.25, -2) {};
		\node [style=none] (99) at (6, -2) {};
		\node [style=none] (100) at (8.25, -2.5) {};
		\node [style=none] (101) at (5.25, -4) {};
		\node [style=none] (102) at (6, -4) {};
		\node [style=none] (103) at (8.25, -4) {};
		\node [style=none] (104) at (10.25, -2.75) {};
		\node [style=none] (105) at (10.25, -2) {};
		\node [style=none] (106) at (10.25, -3.5) {};
		\node [style=none] (107) at (10.25, -3.75) {};
		\node [style=none] (108) at (9.25, -2.5) {};
		\node [style=none] (109) at (9.25, -3.75) {};
		\node [style=none] (110) at (6.75, -4) {$\cdots$};
		\node [style=none] (111) at (6.75, -2) {$\cdots$};
		\node [style=none] (112) at (13.5, 0) {$\Large{=}$};
	\end{pgfonlayer}
	\begin{pgfonlayer}{edgelayer}
		\draw (11.center) to (2.center);
		\draw (12.center) to (3.center);
		\draw (4.center) to (13.center);
		\draw (10.center) to (7.center);
		\draw (9.center) to (6.center);
		\draw (8.center) to (5.center);
		\draw (15.center) to (14.center);
		\draw (16.center) to (17.center);
		\draw (18.center) to (19.center);
		\draw [bend left=90, looseness=2.00] (18.center) to (15.center);
		\draw [bend left=90, looseness=2.25] (19.center) to (13.center);
		\draw (79.center) to (70.center);
		\draw (80.center) to (71.center);
		\draw (72.center) to (81.center);
		\draw (78.center) to (75.center);
		\draw (77.center) to (74.center);
		\draw (76.center) to (73.center);
		\draw (101.center) to (92.center);
		\draw (102.center) to (93.center);
		\draw (94.center) to (103.center);
		\draw (100.center) to (97.center);
		\draw (99.center) to (96.center);
		\draw (98.center) to (95.center);
		\draw (105.center) to (104.center);
		\draw (106.center) to (107.center);
		\draw (108.center) to (109.center);
		\draw [bend left=270, looseness=2.50] (108.center) to (100.center);
		\draw [bend right=90, looseness=3.00] (109.center) to (107.center);
	\end{pgfonlayer}
\end{tikzpicture}
\end{equation}
\begin{equation}\label{eq:usingid} \begin{tikzpicture}
	\begin{pgfonlayer}{nodelayer}
		\node [style=multi function] (0) at (7.5, -5) {$Di$};
		\node [style=1function] (1) at (11.75, -5) {Id$_W$};
		\node [style=none] (2) at (6, -5.5) {};
		\node [style=none] (3) at (7, -5.5) {};
		\node [style=none] (4) at (9.25, -5.5) {};
		\node [style=none] (5) at (6, -4.75) {};
		\node [style=none] (6) at (7, -4.75) {};
		\node [style=none] (7) at (9.25, -4.75) {};
		\node [style=none] (8) at (6, -4) {};
		\node [style=none] (9) at (7, -4) {};
		\node [style=none] (10) at (9.25, -4) {};
		\node [style=none] (11) at (6, -6.25) {};
		\node [style=none] (12) at (7, -6.25) {};
		\node [style=none] (13) at (9.25, -6.25) {};
		\node [style=none] (14) at (11.75, -4.75) {};
		\node [style=none] (15) at (11.75, -4) {};
		\node [style=none] (16) at (11.75, -5.5) {};
		\node [style=none] (17) at (11.75, -6.25) {};
		\node [style=none] (18) at (10.5, -4) {};
		\node [style=none] (19) at (10.5, -6.25) {};
		\node [style=none] (20) at (8, -6) {$\cdots$};
		\node [style=none] (21) at (8, -4) {$\cdots$};
		\node [style=none] (22) at (13, -4) {};
		\node [style=none] (23) at (13, -6.25) {};
		\node [style=none] (24) at (13.75, -5) {$\Large{=}$};
		\node [style=multi function] (25) at (17.25, -5) {$Di$};
		\node [style=none] (26) at (15.75, -5.5) {};
		\node [style=none] (27) at (16.75, -5.5) {};
		\node [style=none] (28) at (19, -5.5) {};
		\node [style=none] (29) at (15.75, -4.75) {};
		\node [style=none] (30) at (16.75, -4.75) {};
		\node [style=none] (31) at (19, -4.75) {};
		\node [style=none] (32) at (15.75, -4) {};
		\node [style=none] (33) at (16.75, -4) {};
		\node [style=none] (34) at (19, -4) {};
		\node [style=none] (35) at (15.75, -6.25) {};
		\node [style=none] (36) at (16.75, -6.25) {};
		\node [style=none] (37) at (19, -6.25) {};
		\node [style=none] (38) at (17.75, -6) {$\cdots$};
		\node [style=none] (39) at (17.75, -4) {$\cdots$};
		\node [style=none] (40) at (20.5, -4) {};
		\node [style=none] (41) at (20.5, -6.25) {};
	\end{pgfonlayer}
	\begin{pgfonlayer}{edgelayer}
		\draw (11.center) to (2.center);
		\draw (12.center) to (3.center);
		\draw (4.center) to (13.center);
		\draw (10.center) to (7.center);
		\draw (9.center) to (6.center);
		\draw (8.center) to (5.center);
		\draw (15.center) to (14.center);
		\draw (16.center) to (17.center);
		\draw (18.center) to (19.center);
		\draw [bend right=90, looseness=1.50] (19.center) to (17.center);
		\draw [bend left=270, looseness=1.25] (18.center) to (10.center);
		\draw [bend left=90, looseness=1.25] (15.center) to (22.center);
		\draw (22.center) to (23.center);
		\draw [bend left=90] (23.center) to (13.center);
		\draw (35.center) to (26.center);
		\draw (36.center) to (27.center);
		\draw (28.center) to (37.center);
		\draw (34.center) to (31.center);
		\draw (33.center) to (30.center);
		\draw (32.center) to (29.center);
		\draw [bend left=270, looseness=1.25] (40.center) to (34.center);
		\draw (40.center) to (41.center);
		\draw [bend left=90, looseness=1.50] (41.center) to (37.center);
	\end{pgfonlayer}
\end{tikzpicture}
\end{equation}
\end{enumerate}
\end{definition}
As the notations above suggest we will consider equivalent diagrams as equal in what follows. 

The \emph{product} $Di_1\star Di_2$ of two diagrams is formed by placing $Di_1$ to the left of $Di_2$. 
By considering the interpretation of diagrams as morphisms, if $Di_1:W^{\ot a}\to W^{\ot b}$ and $Di_2:W^{\ot c}\to W^{\ot d}$ the product $Di_1\star Di_2$ represents the tensor product $Di_1\ot Di_2:W^{\ot (a+c)}\to W^{\ot (b+d)}$.

By rearranging the boxes, which is possible by the permutation relation, every diagram is equivalent to a diagram in which the boxes appear in order, where $x_1<x_2<\ldots<x_r<\Id_W$.
More precisely, every diagram of degree $(p,q)$ can be brought to a diagram of the form
\begin{equation}\begin{tikzpicture}
	\begin{pgfonlayer}{nodelayer}
		\node [style=multi function] (0) at (-5, -1.75) {$\tau$};
		\node [style=multi function] (1) at (-5, -0.25) {$\overline{Di}$};
		\node [style=multi function] (2) at (-5, 1.25) {$\sigma$};
		\node [style=none] (3) at (-7, -2.25) {};
		\node [style=none] (4) at (-6.5, -2.25) {};
		\node [style=none] (5) at (-4.5, -2.25) {};
		\node [style=none] (6) at (-3.25, -2.25) {};
		\node [style=none] (7) at (-4, -2.25) {};
		\node [style=none] (8) at (-7, 1.5) {};
		\node [style=none] (9) at (-6.5, 1.5) {};
		\node [style=none] (10) at (-4.5, 1.5) {};
		\node [style=none] (11) at (-3.25, 1.5) {};
		\node [style=none] (12) at (-4, 1.5) {};
		\node [style=none] (13) at (-0.75, 1.5) {};
		\node [style=none] (14) at (0.5, 1.5) {};
		\node [style=none] (15) at (0, 1.5) {};
		\node [style=none] (16) at (-0.75, -2.25) {};
		\node [style=none] (17) at (0.5, -2.25) {};
		\node [style=none] (18) at (0, -2.25) {};
		\node [style=none] (19) at (-5.75, 2) {$\cdots$};
		\node [style=none] (20) at (-5.75, -2.75) {$\cdots$};
		\node [style=none] (21) at (-6.5, 3.25) {};
		\node [style=none] (22) at (-7, 3.25) {};
		\node [style=none] (24) at (-7, -4) {};
		\node [style=none] (25) at (-6.5, -4) {};
		\node [style=none] (26) at (-3.5, 2) {$\cdots$};
		\node [style=none] (27) at (-3.5, -2.75) {$\cdots$};
		\node [style=none] (28) at (-7, 0.25) {};
		\node [style=none] (29) at (-6.5, 0.25) {};
		\node [style=none] (30) at (-4.5, 0.25) {};
		\node [style=none] (31) at (-4, 0.25) {};
		\node [style=none] (32) at (-3.25, 0.25) {};
		\node [style=none] (33) at (-7, 1) {};
		\node [style=none] (34) at (-6.5, 1) {};
		\node [style=none] (35) at (-4.5, 1) {};
		\node [style=none] (36) at (-4, 1) {};
		\node [style=none] (37) at (-3.25, 1) {};
		\node [style=none] (38) at (-7, -0.75) {};
		\node [style=none] (39) at (-6.5, -0.75) {};
		\node [style=none] (40) at (-4.5, -0.75) {};
		\node [style=none] (41) at (-4, -0.75) {};
		\node [style=none] (42) at (-3.25, -0.75) {};
		\node [style=none] (43) at (-7, -1.5) {};
		\node [style=none] (44) at (-6.5, -1.5) {};
		\node [style=none] (45) at (-4.5, -1.5) {};
		\node [style=none] (46) at (-4, -1.5) {};
		\node [style=none] (47) at (-3.25, -1.5) {};
		\node [style=none] (48) at (-5.75, -1) {$\cdots$};
		\node [style=none] (49) at (-5.75, 0.5) {$\cdots$};
		\node [style=none] (50) at (-9.75, -0.25) {$c(j,\sigma,\tau,n_1,\ldots,n_r,m) = $};
	\end{pgfonlayer}
	\begin{pgfonlayer}{edgelayer}
		\draw [bend left=90, looseness=1.25] (10.center) to (14.center);
		\draw (14.center) to (17.center);
		\draw [bend left=90, looseness=1.25] (17.center) to (5.center);
		\draw [bend left=90, looseness=1.25] (12.center) to (15.center);
		\draw (15.center) to (18.center);
		\draw [bend left=90, looseness=1.25] (18.center) to (7.center);
		\draw [bend left=90, looseness=1.25] (11.center) to (13.center);
		\draw (13.center) to (16.center);
		\draw [bend left=90, looseness=1.25] (16.center) to (6.center);
		\draw (21.center) to (9.center);
		\draw (22.center) to (8.center);
		\draw (3.center) to (24.center);
		\draw (25.center) to (4.center);
		\draw (33.center) to (28.center);
		\draw (34.center) to (29.center);
		\draw (35.center) to (30.center);
		\draw (36.center) to (31.center);
		\draw (37.center) to (32.center);
		\draw (42.center) to (47.center);
		\draw (41.center) to (46.center);
		\draw (40.center) to (45.center);
		\draw (39.center) to (44.center);
		\draw (38.center) to (43.center);
	\end{pgfonlayer}
\end{tikzpicture}
\end{equation}
where $\sigma\in S_{p+j}$, $\tau\in S_{q+j}$, the number of connections of input and output strings is $j$, and $$\overline{Di} = x_1^{\star n_1}\star\cdots\star x_r^{\star n_r}\star \Id_W^{\star m}.$$ This diagram represents the linear transformation
$$ev^j(L^{(p+j)}_{\sigma}(x_1^{\ot n_1}\ot \cdots\ot x_r^{\ot n_r}\ot \Id_W^{\ot m})L^{(q+j)}_{\tau}).$$ It is possible that two tuples of data $(j,\sigma,\tau,n_1,\ldots,n_r,m)$ and $(j',\sigma',\tau',n_1',\ldots,n_r',m')$ will define equivalent diagrams. This can happen by the equivalence relation of Definition \ref{def:eqDi}, or by altering both $\sigma$ and $\tau$ by permuting the last $j$ connected strings in the same way. We summarize this in the following lemma:
\begin{lemma}\label{lem:eqDi}
Two diagrams are equivalent if and only if one can be obtained from the other by the following rules:
\begin{enumerate}
\item If $\mu\in S_j$ then $c(j,\iota_1(\mu)\sigma,\tau,n_1,\ldots,n_r,m) = c(j,\sigma,\tau\iota_2(\mu),n_1,\ldots,n_r,m)$
where $\iota_1:S_j\to S_{p+j}$ is given by $S_j\stackrel{1\times Id}{\to}S_p\times S_j\to S_{p+j}$ and $\iota_2:S_j\to S_{q+j}$ is defined similarly.
\item If $m>0$, $j>0$, $\sigma(p+j) = p+j$ and $\tau(q+j)\neq q+j$ then $$c(j,\sigma,\tau,n_1,\ldots,n_r,m) = c(j-1,\sigma, (\tau(q+j),q+j)\tau,n_1,\ldots,n_r,m-1).$$ Similarly, if $m>0,j>0$, $\sigma(p+j)\neq p+j$ and $\tau(q+j)=q+j$ then 
$$c(j,\sigma,\tau,n_1,\ldots,n_r,m) = c(j-1,(\sigma(p+j),p+j)\sigma,\tau,n_1,\ldots,n_r,m-1).$$ 
\item Write $\alpha_1= \alpha^{(p_1,\ldots, p_r,1)}_{(n_1,\ldots,n_r,m)}: S_{n_1}\times\cdots\times S_{n_r}\times S_m\to S_{p+j}$ and $\alpha_2 = \alpha^{(q_1,\ldots, q_r,1)}_{(n_1,\ldots,n_r,m)}: S_{n_1}\times\cdots\times S_{n_r}\times S_m\to S_{q+j}$. Then for every $(\mu_i)\in S_{n_1}\times\cdots\times S_{n_r}\times S_m$ we have 
$$c(j,\sigma\alpha_1((\mu_i)),\tau,n_1,\ldots,n_r,m) = c(j,\sigma,\alpha_2((\mu_i))\tau,n_1,\ldots, n_r,m)$$
\end{enumerate}
\end{lemma}
\begin{proof}
The first part follows from the discussion before the lemma. The last two parts correspond to the two parts of Definition \ref{def:eqDi}.
\end{proof}
We finish this section with a formula for the product of two diagrams. 
Assume that we have two diagrams given by the data $c(j,\sigma,\tau,n_1,\ldots,n_r,m)$ and $c(j',\sigma',\tau',n_1',\ldots,n_r',m')$. Consider the unordered partitions $$\la_1 = (n_1p_1,n_2p_2,\ldots n_rp_r,m,n_1'p_1,n_2'p_2,\ldots,n_r'p_r,m')\text{ and }$$
$$\la_2 = (n_1q_1,n_2q_2,\ldots n_rq_r,m,n_1'q_1,n_2'q_2,\ldots,n_r'q_r,m').$$
Define $\mu\in S_{2r+2}$ by the formula $$\mu(a) = \begin{cases} \frac{a}{2} + r+1 \text{ if } a \text{ is even} \\
\frac{a+1}{2} \text{ if } a \text{ is odd}\end{cases}$$
Write $\mu_1 = \Omega^{\la_1}(\mu)$ and $\mu_2 = \Omega^{\la_2}(\mu)$. So $\mu_1\in S_P$ where $P= \sum (n_i+n_i')p_i + m+m'$ and $\mu_2\in S_{Q}$ where $Q= \sum_i(n_i+n_i')q_i + m+m'$. 
The following formula holds:
$$\xini\ot \Id_W^{\ot m}\ot x_1^{\ot n_1'}\ot\cdots\ot x_r^{\ot n_r'}\ot \Id_W^{\ot m'}= L^{(P)}_{\mu_1}(x_1^{\ot (n_1+n_1')}\ot\cdots\ot x_r^{\ot (n_r+n'_r)}\ot \Id_W^{\ot (m+m')})L^{(Q)}_{\mu_2^{-1}}$$
So by considering the linear maps which the diagrams represent we get that 
$$c(j,\sigma,\tau,n_1,\ldots,n_r,m)\star c(j',\sigma',\tau',n_1',\ldots,n_r',m')$$ represents the linear map 
$$ev^j(L^{(p+j)}_{\sigma}(\xini\ot \Id_W^{\ot m})L^{(q+j)}_{(\tau)})\ot ev^{j'}(L^{(p'+j')}_{\sigma}(x_1^{\ot n_1'}\ot\cdots\ot x_r^{\ot n_r'}\ot \Id_W^{\ot m'})L^{(q'+j')}_{(\tau')}) = $$
$$(ev^j\ot ev^{j'})(L^{(P)}_{(\sigma,\sigma')}L^{(P)}_{\mu_1}(x_1^{\ot (n_1+n_1')}\ot\cdots\ot x_r^{\ot (n_r+n_r')}\ot \Id_W^{\ot (m+m')})L^{(Q)}_{\mu_2^{-1}}L^{(Q)}_{(\tau,\tau')})$$
We next use the fact that for any $T\in W^{P,Q}$ it holds that
$$(ev^j\ot ev^{j'})(T) = ev^{j+j'}(L^{(P)}_{\mu_3^{-1}}TL^{(Q)}_{\mu_4})$$ where 
$\mu_3 = \Omega^{(p,j,p',j')}((23))\in S_P$ and $\mu_4 = \Omega^{(q,j,q',j')}((23))$.
Here $p=\sum_i n_ip_i + m-j$, $p'=\sum_in'_ip_i+m'-j'$, and $q$ and $q'$ are defined similarly.  
Substituting this into the previous equation we get the formula
$$c(j,\sigma,\tau,n_1,\ldots,n_r,m)\star c(j',\sigma',\tau',n_1',\ldots,n_r',m')= $$
\begin{equation}\label{eq:diagramproduct}c(j+j',\mu_3^{-1}(\sigma,\sigma')\mu_1,\mu_2^{-1}(\tau,\tau')\mu_4,n_1+n_1',\ldots, n_r+n_r',m+m')\end{equation}
 where $\mu_1,\mu_2,\mu_3,\mu_4$ are given as above.
It is much easier, though, to just describe the product of diagrams graphically by placing them one next to the other.  
\section{Second definition of $\Ainf$ and the algebra $\winf$}\label{sec:ainfwinf}
Let $Con^{p,q}$ be the free vector space generated by all the equivalence classes of diagrams of degree $(p,q)$. The $\star$-product equips the direct sum $Con=\bigoplus_{p,q}Con^{p,q}$ with a structure of an $\N^2$ graded algebra. This algebra will come into play in Section \ref{sec:axid} where we will study the ideal arising from axioms. We will concentrate now on the subalgebra of trivial degree.
\begin{definition} We write $\winf = Con^{0,0}$.
\end{definition}
The algebra $\winf$ is a subalgebra of $Con$. It is commutative, as can be seen by a direct calculation using Formula \ref{eq:diagramproduct} for the $\star$ product of two diagrams. The intuitive explanation for this is that if $Di_1$ and $Di_2$ are two closed diagrams then $Di_1\star Di_2$ can be obtained from $Di_2\star Di_1$ just by moving the boxes which appear in $Di_2$ to the right, as there are no free input or output strings that should be fixed. This is not true for general diagrams. 
The algebra $\winf$ has an extra grading arising from the number of boxes. 
\begin{definition}
We define the box-degree of $c(j,\sigma,\tau,n_1,\ldots,n_r,m)\in\winf$ to be $(n_1,\ldots,n_r)\in \N^r$.
\end{definition}
When we will concentrate on $\winf$ instead of $Con$, we will just refer to the box-degree as the degree.
We do not include the number $m$ in the grading, as it can vary between equivalent diagrams. Definition \ref{def:eqDi} shows that the box-degree notion is well defined.

\begin{definition}\label{def:basicinvariants} Assume that $n=\sum_ip_in_i + m = \sum_i q_in_i + m$. We write 
$$p(n,\sigma,n_1,\ldots,n_r,m) = c(n,\sigma,\Id,n_1,\ldots,n_r,m).$$ We say that the closed diagram $p(n,\sigma,n_1,\ldots,n_r,m)$ is a polynomial invariant of structures of type $((p_i,q_i))$. 
\end{definition}
Notice that indeed, every $p(n,\sigma,n_1,\ldots,n_r,m)$ can be evaluated on any structure of type $((p_i,q_i))$, no matter of what dimension, to give a scalar invariant. This scalar is $$\Tr(L_{\sigma}^{(n)}(\xini\ot \Id_W^{\ot m})).$$ In case $m=0$ and we restrict our attention to structures of a fixed dimension $d$ we get the basic polynomial invariants of Section \ref{sec:proofthm1}. We will think of the symbol $p(n,\sigma,n_1,\ldots,n_r)$ as representing both a diagram and a polynomial invariant. 

By taking only the elements $c(n,\sigma,\tau,n_1,\ldots,n_r,m)$ in which $\tau=\Id$ we impose no restriction. Indeed, by Lemma \ref{lem:eqDi} we know that for a general permutation $\tau\in S_n$ we have $c(n,\sigma,\tau,n_1,\ldots,n_r,m) = c(n,\tau\sigma,\Id,n_1,\ldots,n_r,m) = p(n,\tau\sigma,n_1,\ldots,n_r,m)$.

\begin{definition} An invariant $p(n,\sigma,n_1,\ldots,n_r,m)$ is called connected, or irreducible, if the corresponding diagram is connected. Equivalently- an invariant is connected if it cannot be written non-trivially as the product of two invariants. 
\end{definition}

\begin{definition}\label{def:augmented} We define the dimension invariant as $p(1,\Id,0,0,\ldots,0,1)$. We denote it by $D$. We define an invariant to be unaugmented if it is equal to an invariant of the form $p(n,\sigma,n_1,\ldots,n_r,0)$. Otherwise it is augmented.
\end{definition}
The dimension invariant is represented by the following diagram:
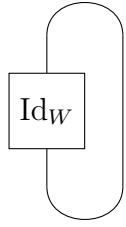
\begin{equation}\label{fig:identity}
\begin{tikzpicture}
	\begin{pgfonlayer}{nodelayer}
		\node [style=1function] (0) at (3, 1.25) {$\Id_W$};
		\node [style=none] (1) at (3, 0.75) {};
		\node [style=none] (2) at (3, 0.25) {};
		\node [style=none] (3) at (3, 1.75) {};
		\node [style=none] (4) at (3, 2.25) {};
		\node [style=none] (5) at (4, 2.25) {};
		\node [style=none] (6) at (4, 0.25) {};
	\end{pgfonlayer}
	\begin{pgfonlayer}{edgelayer}
		\draw (4.center) to (3.center);
		\draw [bend left=90, looseness=1.50] (4.center) to (5.center);
		\draw (5.center) to (6.center);
		\draw [bend left=90, looseness=1.50] (6.center) to (2.center);
		\draw (2.center) to (1.center);
	\end{pgfonlayer}
\end{tikzpicture}
\end{equation}
We call this invariant the dimension as it represents the trace of the identity map of $W$, which is indeed the dimension of $W$. 
\begin{proposition}
Every invariant in $\winf$ can be written uniquely (up to reordering) as a product of connected invariants. The only augmented connected invariant is $D$.
\end{proposition}
\begin{proof}
The first part of the proposition follows by considering the connected components of the diagram of every invariant, as they are uniquely defined. Notice that reducing invariants using the identity-reduction move in Definition \ref{def:eqDi} does not change the number of irreducible components of the diagram.  
Let now $p(n,\sigma,n_1,\ldots,n_r,m)$ be a connected invariant. We will show that if this invariant is not $D$ and $m\neq 0$ then it is equivalent to an invariant with a smaller value of $m$. By induction this proves the result.
For this, consider the diagram of $p(n,\sigma,n_1,\ldots,n_r,m)$. Consider any of the $\Id_W$ boxes. If the input string and the output string of this box are connected to each other then we get a connected component of the diagram which looks like Figure \ref{fig:identity}. We assumed that the invariant is connected and different from $D$, so this is impossible. 
We can thus use the identity-reduction move of \ref{def:eqDi} to get rid of this $\Id_W$-box from the diagram, thus reducing the value of $m$ by one, as required.
\end{proof}
We recall now Definition \ref{def:seconddefkx}. 
\begin{definition}[The algebra $\Ainf$- second definition]
We write $\Ainf\subseteq \winf$ for the subalgebra spanned by the unaugmented invariants. This is a subalgebra indeed, as the product of two diagrams which do not contain any apperance of $\Id_W$ is again such a diagram. 
For unaugmented invariants we write $$p(n,\sigma,n_1,\ldots,n_r) := p(n,\sigma,n_1,\ldots,n_r,0)$$ where $\sum_i n_ip_i = \sum_i n_iq_i = n$. 
\end{definition}

\begin{lemma} The two definitions of $\Ainf$ are equivalent. 
\end{lemma}
\begin{proof}
The isomorphism between the two algebras is the tautological one, given by sending $p(n,\sigma,n_1,\ldots,n_r)$ to $p(n,\sigma,n_1,\ldots,n_r)$ and extending by linearity. By the equivalence relation on diagrams from Lemma \ref{lem:eqDi} and by Relation $R1$ in Proposition \ref{prop:invpolys} we get that this is well defined. This isomorphism preserves the multiplication indeed, since taking the product of diagrams corresponds to taking the tensor product of the linear maps they represent. In case of maps $K\to K$, this corresponds to just multiplication of scalars, which is the product in the first definition of $\Ainf$. 
\end{proof}

The discussion about augmented and unaugmented invariants has the following corollary:
\begin{corollary}\label{cor:ainfwinf}
The algebra $\Ainf$ is a polynomial algebra. Its set of variables are all the unaugmented connected invariants. The algebra $\winf$ is isomorphic to $\Ainf\ot K[D]$. 
\end{corollary}
We will think of both algebras as algebras of invariants of structures of type $((p_i,q_i))_i$. The only difference is that $\Ainf$ ``cannot see'' the dimension of $W$, while $\winf$ can. 

We summarize the structure of $\Ainf$ in the following proposition. All of the claims follow from the previous claims about the elements $c(j,\sigma,\tau,n_1,\ldots,n_r,m)$. 
\begin{proposition}\label{prop:ainfform}
The algebra $\Ainf$ is a polynomial algebra on the connected unaugmented polynomial invariants, and has a basis given by all unaugmented polynomial invariants. 
The multiplication in $\Ainf$ is given by the formula 
$$p(n,\sigma,n_1,\ldots,n_r)\star p(n',\sigma',n_1',\ldots,n_r') = p(n+n',\mu_2^{-1}(\sigma,\sigma')\mu_1,n_1+n_1',\ldots,n_r+n_r')$$ where $\mu_1$ and $\mu_2$ are as defined in the last part of Section \ref{sec:diagrams} (the permutations $\mu_3$ and $\mu_4$ are trivial here).
Two polynomial invariants $p(n,\sigma,n_1,\ldots,n_r)$ and $p(n',\sigma',n_1',\ldots,n_r')$ are equal if and only if $n=n', n_1=n_1',\ldots, n_r=n_r'$ and $$\sigma = \alpha^{(q_i)}_{(n_i)}((\mu_i))^{-1}\sigma'\alpha^{(p_i)}_{(n_i)}((\mu_i))$$ for some $(\mu_i)\in S_{n_1,\ldots,n_r}$, where $\alpha^{(q_i)}_{(n_i)}$ and $\alpha^{(p_i)}_{(n_i)}$ are given in Definition \ref{def:alpha}. 
\end{proposition}
The above proposition gives an algebraic formula for the multiplication of two basic invariants. It is, however, easier to think of this multiplication as taking disjoint union of diagrams.

\section{The bialgebra structures on $\Ainf$}\label{sec:bialg}
The fact that elements of $\Ainf$ are invariant polynomials which can be evaluated on structures of all dimensions enriches the structure of $\Ainf$. We exhibit here two bialgebra structures on $\Ainf$, one coming from forming the direct sum of structures, the other from forming the tensor product of structures. We begin with the following definition:
\begin{definition} Let $(W_1,(x_i)), (W_2,(y_i))$ be two algebraic structures of type $((p_i,q_i))$. The direct sum of these structures is the structure $(W_1\oplus W_2,(x_i\oplus y_i))$ where $x_i\oplus y_i\in W_1^{p_i,q_i}\oplus W_2^{p_i,q_i}\subseteq (W_1\oplus W_2)^{p_i,q_i}$. The tensor product of the structures is the structure $(W_1\ot W_2,(x_i\ot y_i))$, where $x_i\ot y_i\in W_1^{p_i,q_i}\ot W_2^{p_i,q_i}\cong (W_1\ot W_2)^{p_i,q_i}$.
\end{definition}
If the structure we are considering is that of associative algebras then the direct sum and tensor product we defined here are the usual direct sum and tensor product of algebras. 

If $(W_1,(x_i))$ and $(W_2,(y_i))$ are two structures with closed orbit, then we can consider the unique closed orbits in the closures of the orbits of $(W_1\oplus W_2, (x_i\oplus y_i))$ and of $(W_1,\ot W_2,(x_i\ot y_i))$. This induces operations $\oplus:X\times X\to X$ and $\ot:X\times X\to X$. The following lemma holds:
\begin{lemma}\label{lem:products}
The operations $\oplus$ and $\ot$ define associative operations $\oplus,\ot:X\times X\to X$.
Both operations have a neutral element. The neutral element for $\oplus$ is the zero-dimensional structure with the zero tensors. The neutral element for $\ot$ is the structure $(K,(x_i))$ where $x_i=1\in K^{p,q}\cong K$. We denote this structure by $(K,(1))$. 
\end{lemma}
We would like to show that these operations define two bialgebra structures on $\Ainf$. Notice that for $m=\oplus$ or $\ot$ we have the following diagram:
$$\xymatrix{\Ainf\ar[r]^{\Xi} & K^X \ar[d]^{m^*}\\
 & K^{X\times X} \\ 
\Ainf\ot \Ainf\ar[r]^{\Xi\ot\Xi} & K^X\ot K^X\ar[u]^{\beta}}
$$
where $\beta(f\ot g)(x,y) = f(x)g(y)$.
Since $\beta(\Xi\ot\Xi)$ is an injective map it will be enough to show that the image of $m^*\Xi$ is contained in the image of $\Ainf\ot\Ainf$ under $\beta(\Xi\ot\Xi)$. This will imply that we have a map $\Ainf\to \Ainf\ot\Ainf$ which makes the diagram above commutative. The multiplicativity of the resulting map $\Ainf\to\Ainf\ot\Ainf$ follows from the multiplicativity of all the involved maps. The coassicativity follows from the associativity of $m$. 

We begin with $m=\ot$. For this, we calculate:
$$\ot^*\Xi(p(n,\sigma,n_1,\ldots,n_r))((W_1,(x_i)),(W_2,(y_i))) = $$
$$\Tr_{W_1^{\ot n}\ot W_2^{\ot n}}(L_{\sigma}^{(n)}(x_1\ot y_1)^{\ot n_1}\ot\cdots\ot (x_r\ot y_r)^{\ot n_r}) = $$
$$\Tr_{W_1^{\ot n}}(L^{(n)}_{\sigma}\xini)\cdot 
\Tr_{W_2^{\ot n}}(L^{(n)}_{\sigma}y_1^{\ot n_1}\ot\cdots\ot y_r^{\ot n_r}) = $$
$$p(n,\sigma,n_1,\ldots,n_r)((W_1,(x_i)))p(n,\sigma,n_1,\ldots,n_r)((W_2,(y_i)))=$$
$$\beta(\Xi\ot\Xi)(p(n,\sigma,n_1,\ldots n_r)\ot p(n,\sigma,n_1,\ldots n_r)).$$
This shows that $\ot$ defines a bialgebra strcuture which we denote by $$\Delot:\Ainf\to \Ainf\ot\Ainf.$$ Moreover, it shows us that with respect to this coproduct all the basic invariants are group-like elements (i.e. they satisfy the equation $\Delot(g) = g\ot g$). This shows that $\Delot$ defines on $\Ainf$ a bialgebra structure which is not a Hopf algebra structure, assuming that $\Ainf$ is not trivial. Indeed, if $\Ainf$ had an antipode $S$ with respect to $\Delot$ then the antipode equation would have implied that $S(g)=g^{-1}$ for every group-like element. Since $\Ainf$ is graded and all the basic invariants are of positive degree this is impossible. The counit $\epsilon_{\ot}$ for this bialgebra structure is given by evaluation on the one dimensional structure $(K,(1))$ from Lemma \ref{lem:products}. We have 
$$\epsilon_{\ot}(p(n,\sigma,n_1,\ldots,n_r))=1$$ for every basic invariant. 

Consider now the direct sum operation. We will show that it defines another bialgebra structure on $\Ainf$. We begin with the following lemma:
\begin{lemma}
Let $p(n,\sigma,n_1,\ldots,n_r)$ be a basic invariant, where $n=\sum_i p_in_i = \sum_i q_in_i$ and $\sigma\in S_n$.
The diagram of this basic invariant contains $n_1+n_2+\cdots + n_r$ boxes and $n$ strings connecting them.
Number these strings by $\{1,\ldots,n\}$. Denote by $I_i$ the set of input strings for the $i$-th box and by $J_i$ the set of output strings of the $i$-th box. 
Then $p(n,\sigma,n_1,\ldots,n_r)$ is irreducible if and only if the following condition holds: 
there is no proper partition $\{1,\ldots,n\} = X\sqcup Y$ such that for every $i$ $I_i\cup \sigma(J_i)\subseteq X$ or $I_i\cup \sigma(J_i)\subseteq Y$.
\end{lemma}
\begin{proof}
If the diagram of $p(n,\sigma,n_1,\ldots,n_r)$ is not connected, write $\{1,\ldots,n\} = X\sqcup Y$ where $X$ contains the numbers of the input strings of one component, and $Y=\{1,\ldots,n\}\backslash X$. This will give us the desired decomposition.
In the other direction, if we have such a partition, we get a non-trivial decomposition of the diagram of $p(n,\sigma,n_1,\ldots,n_r)$ into different connected components. 
\end{proof}

 
Since the irreducible basic invariants generate $\Ainf$ it will be enough to prove that every irreducible basic invariant $p(n,\sigma,n_1,\ldots, n_r)$ has image inside $\beta(\Xi\ot\Xi)(\Ainf\ot\Ainf)$. We will show that in fact all the irreducible basic invariants are primitive with respect to $\Delta$. 

To do so, we consider two structures
$(W_1,(x_i))$ and $(W_2,(y_i))$ and their direct sum 
$(W,(z_i))=(W_1\oplus W_2,(x_i\oplus y_i))$. 
We write $P_1:W\to W_1\to W$ and $P_2:W\to W_2\to W$ for the natural projections followed by the natural embeddings. 
We calculate: 
$$\oplus^*\Xi(p(n,\sigma,n_1,\ldots n_r))((W_1,(x_i)),(W_2,(y_i))) = 
p(n,\sigma,n_1,\ldots, n_r)(W,(z_i)) = $$
$$Tr_{W^{\ot n}}(L^{(n)}_{\sigma}(z_1^{\ot n_1}\ot \cdots\ot z_r^{\ot n_r}))=$$
$$\sum_{j_1,\ldots,j_n=1}^{2} Tr_{W^{\ot n}}(L^{(n)}_{\sigma}(z_1^{\ot n_1}\ot \cdots\ot z_r^{\ot n_r})(P_{j_1}\ot P_{j_2}\ot\cdots\ot P_{j_n})).$$
We claim the following:
\begin{lemma}
If $(j_1,\ldots, j_n)\neq (1,1,\ldots,1)$ or $(2,2,\ldots,2)$ then $
Tr_{W^{\ot n}}(L^{(n)}_{\sigma}(z_1^{\ot n_1}\ot \cdots\ot z_r^{\ot n_r})(p_{j_1}\ot p_{j_2}\ot\cdots\ot p_{j_n}))=0$. 
\end{lemma}
\begin{proof}
Assume that $(j_1,\ldots, j_n)\neq (1,1,\ldots, 1)$ or $(2,2,\ldots,2)$ and that 
$Tr_{W^{\ot n}}(L^{(n)}_{\sigma}(z_1^{\ot n_1}\ot \cdots\ot z_r^{\ot n_r})(P_{j_1}\ot P_{j_2}\ot\cdots\ot P_{j_n}))\neq 0$. 
Write $$T= L^{(n)}_{\sigma}(z_1^{\ot n_1}\ot \cdots\ot z_r^{\ot n_r})(P_{j_1}\ot P_{j_2}\ot\cdots\ot P_{j_n}).$$
We will think of the tuple $(j_1,\ldots, j_n)$ as a labeling of the strings which appear in the diagram which corresponds to $T$.

It holds that $z_i = x_i\oplus y_i\in W_1^{p_i,q_i}\oplus W_2^{p_i,q_i}$. 
This implies that if we have two strings going into the same box in the diagram which represents the morphism $T$ then they must have the same label, as otherwise $T=0$ and $\Tr(T)=0$, contradicting the assumption. 
In a similar way, two strings going out of the same box must have the same label.
This already implies that all the strings in the same connected component of the diagram which represents $\Tr(T)$ must be equal. Since we have assumed that $p(n,\sigma,n_1,\ldots,n_r)$ is irreducible, we have a single connected component, and thus $\Tr(T)$ can be non-zero only if $(j_1,\ldots j_n) = (1,1,\ldots,1)$ or $(2,2,\ldots,2)$. 
\end{proof}
The above proof thus shows that 
$$\oplus^*\Xi(p(n,\sigma,n_1,\ldots, n_r))((W_1,(x_i)),(W_2,(y_i))) =  $$
$$p(n,\sigma,n_1,\ldots, n_r)((W_1,(x_i))) + p(n,\sigma,n_1,\ldots,n_r)((W_2,(y_i))).$$
when $p(n,\sigma,n_1,\ldots,n_r)$ is an irreducible basic invariant.
This means that $\Delta:\Ainf\to \Ainf\ot\Ainf$ is well defined and gives another bialgebra structure on $\Ainf$. It also means that the irreducible basic invariants are primitive with respect to $\Delta$. 
It is known that such a bialgebra is in fact a Hopf algebra. Indeed, since $\Ainf$ is generated by the irreducible basic invariants, the antipode $S$ is defined by the formula $S(p_1p_2\cdots p_t) = (-1)^tp_1p_2\cdots p_t$ where $p_i$ are irreducible basic invariants. 
We have used here the fact that $S(p)=-p$ for a primitive elements, and that $S$ is multiplicative (and not just anti-multiplicative, since the algebra $\Ainf$ is commutative).

\section{The axioms ideal}\label{sec:axid}
The goal of this section is to discuss the ideals $I_{\T,d}$ and $I_{\T,d}^{\GL_d(K)}$ which appear in the $R3$ relations in Theorem \ref{thm:mainstructure}. 
Axioms for an algebraic structure are often given by stating equality of two diagrams, or, more generally, by stating that a certain linear combination of some constructible maps is zero (e.g. the Jacobi identity for Lie algebras). 
For example, assume that $x_1\in W^{1,2}$. The tensor $x_1$ can be thought of as a multiplication on $W$. the associativity of $x_1$ is equivalent to the vanishing of the following linear combination of diagrams, when interpreted as elements in $W^{1,3}$:
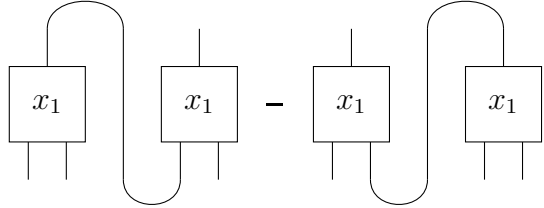
\begin{center}
\begin{tikzpicture}
	\begin{pgfonlayer}{nodelayer}
		\node [style=2function] (0) at (-1, 1) {$x_1$};
		\node [style=2function] (1) at (1, 1) {$x_1$};
		\node [style=none] (2) at (-1.25, 0.5) {};
		\node [style=none] (3) at (-0.75, 0.5) {};
		\node [style=none] (4) at (-1.25, 0) {};
		\node [style=none] (5) at (-0.75, 0) {};
		\node [style=none] (6) at (-1, 1.5) {};
		\node [style=none] (7) at (1, 1.5) {};
		\node [style=none] (8) at (0.75, 0.5) {};
		\node [style=none] (9) at (1.25, 0.5) {};
		\node [style=none] (10) at (0.75, 0) {};
		\node [style=none] (11) at (1.25, 0) {};
		\node [style=none] (12) at (1, 2) {};
		\node [style=none] (13) at (-1, 2) {};
		\node [style=none] (14) at (0, 2) {};
		\node [style=none] (15) at (0, 0) {};
		\node [style=none] (16) at (2, 1) {\Huge - };
		\node [style=2function] (17) at (3, 1) {$x_1$};
		\node [style=2function] (18) at (5, 1) {$x_1$};
		\node [style=none] (19) at (2.75, 0.5) {};
		\node [style=none] (20) at (3.25, 0.5) {};
		\node [style=none] (21) at (2.75, 0) {};
		\node [style=none] (22) at (3.25, 0) {};
		\node [style=none] (23) at (3, 1.5) {};
		\node [style=none] (24) at (5, 1.5) {};
		\node [style=none] (25) at (4.75, 0.5) {};
		\node [style=none] (26) at (5.25, 0.5) {};
		\node [style=none] (27) at (4.75, 0) {};
		\node [style=none] (28) at (5.25, 0) {};
		\node [style=none] (29) at (5, 2) {};
		\node [style=none] (30) at (3, 2) {};
		\node [style=none] (31) at (4, 2) {};
		\node [style=none] (32) at (4, 0) {};
	\end{pgfonlayer}
	\begin{pgfonlayer}{edgelayer}
		\draw (3.center) to (5.center);
		\draw (2.center) to (4.center);
		\draw (8.center) to (10.center);
		\draw (9.center) to (11.center);
		\draw (13.center) to (6.center);
		\draw (12.center) to (7.center);
		\draw [bend left=90, looseness=1.25] (13.center) to (14.center);
		\draw (14.center) to (15.center);
		\draw [bend right=90, looseness=1.50] (15.center) to (10.center);
		\draw (20.center) to (22.center);
		\draw (19.center) to (21.center);
		\draw (25.center) to (27.center);
		\draw (26.center) to (28.center);
		\draw (30.center) to (23.center);
		\draw (29.center) to (24.center);
		\draw (31.center) to (32.center);
		\draw [bend left=270, looseness=1.25] (29.center) to (31.center);
		\draw [bend left=90, looseness=1.50] (32.center) to (22.center);
	\end{pgfonlayer}
\end{tikzpicture}
\end{center}
If $x_2\in W^{1,0}$ then the statement that $x_2:K\to W$ defines a left unit for the multiplication $x_1$ is equivalent to the vanishing of the following linear combination of diagrams of degree $(1,1)$, when interpreted as an element in $W^{1,1}$. 
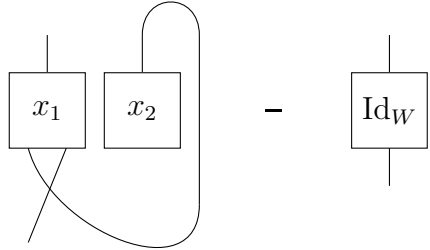
\begin{center}
\begin{tikzpicture}
	\begin{pgfonlayer}{nodelayer}
		\node [style=2function] (0) at (-1, 1) {$x_1$};
		\node [style=none] (2) at (-1.25, 0.5) {};
		\node [style=none] (3) at (-0.75, 0.5) {};
		\node [style=none] (4) at (-1.25, -0.75) {};
		\node [style=none] (6) at (-1, 1.5) {};
		\node [style=none] (13) at (-1, 2) {};
		\node [style=2function] (14) at (0.25, 1) {$x_2$};
		\node [style=none] (15) at (0.25, 1.5) {};
		\node [style=none] (16) at (0.25, 2) {};
		\node [style=none] (17) at (1, 2) {};
		\node [style=none] (18) at (1, -0.25) {};
		\node [style=none] (20) at (2, 1) {\Huge -};
		\node [style=1function] (21) at (3.5, 1) {$\Id_W$};
		\node [style=none] (22) at (3.5, 2) {};
		\node [style=none] (23) at (3.5, 1.5) {};
		\node [style=none] (24) at (3.5, 0.5) {};
		\node [style=none] (25) at (3.5, 0) {};
	\end{pgfonlayer}
	\begin{pgfonlayer}{edgelayer}
		\draw (13.center) to (6.center);
		\draw (15.center) to (16.center);
		\draw [bend left=90, looseness=2.00] (16.center) to (17.center);
		\draw (17.center) to (18.center);
		\draw [in=-75, out=-90, looseness=1.25] (18.center) to (2.center);
		\draw (4.center) to (3.center);
		\draw (22.center) to (23.center);
		\draw (24.center) to (25.center);
	\end{pgfonlayer}
\end{tikzpicture}
\end{center}
Similarly, commutativity of $x_1$ is equivalent to the vanishing of the linear combination 
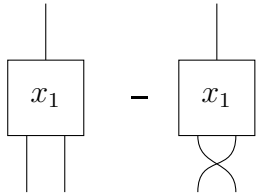
\begin{center}
\begin{tikzpicture}
	\begin{pgfonlayer}{nodelayer}
		\node [style=2function] (0) at (-3, 1.25) {$x_1$};
		\node [style=none] (2) at (-1.75, 1.25) {\Huge -};
		\node [style=none] (3) at (-2.75, 0) {};
		\node [style=none] (4) at (-3.25, 0) {};
		\node [style=none] (5) at (-3, 1.75) {};
		\node [style=none] (9) at (-3, 2.5) {};
		\node [style=none] (10) at (-2.75, 0.75) {};
		\node [style=none] (11) at (-3.25, 0.75) {};
		\node [style=2function] (12) at (-0.75, 1.25) {$x_1$};
		\node [style=none] (13) at (-1, 0) {};
		\node [style=none] (14) at (-0.5, 0) {};
		\node [style=none] (15) at (-0.75, 1.75) {};
		\node [style=none] (16) at (-0.75, 2.5) {};
		\node [style=none] (17) at (-0.5, 0.75) {};
		\node [style=none] (18) at (-1, 0.75) {};
	\end{pgfonlayer}
	\begin{pgfonlayer}{edgelayer}
		\draw (9.center) to (5.center);
		\draw (10.center) to (3.center);
		\draw (11.center) to (4.center);
		\draw (16.center) to (15.center);
		\draw [in=90, out=-90, looseness=1.50] (17.center) to (13.center);
		\draw [in=90, out=-90, looseness=1.25] (18.center) to (14.center);
	\end{pgfonlayer}
\end{tikzpicture}
\end{center}

We thus make the following definition:
\begin{definition}
A theory for structures of type $((p_i,q_i))$ is a subset $\T\subseteq \sqcup_{p,q} Con^{p,q}$.
\end{definition}
We will think of the elements of the theory as those linear combinations that should vanish in models of the theory. To state this explicitly, we give the following definition:

\begin{definition}
If $(W,(x_i))$ is an algebraic structure of type $((p_i,q_i))$ we define the $(p,q)$-realization map $$\Re^{p,q}:Con^{p,q}\to W^{p,q}$$ for $p,q\in \N$ to be the linear map which sends every diagram to its interpretation as a linear map in $W^{p,q}$. 
In particular, $\Re^{0,0}:\winf\to W^{0,0}=K$ is just the character of invariants of $W$. 
\end{definition}
Since equivalent diagrams define identical linear maps the realization maps are well defined. 
\begin{definition}
We say that $W$ is a model for a theory $\T$ if for every $x\in \T\cap Con^{p,q}$ it holds that $\Re^{p,q}(x)=0$ in $W^{p,q}$. 
\end{definition}
The concept here of ``model'' and ``theory'' differs slightly from that of first-order logic. In first order logic only functions $W^{\times a}\to W$ are allowed, and we consider here maps which are assume to be linear between different tensor powers of $W$. 
Fix now a theory $\T$.
\begin{lemma}
Let $W=K^d$ and let $U = U(W)=\bigoplus_{i=1}^r W^{p_i,q_i}$ be as defined in Subsection \ref{subsec:bigalgebra}. The set of all tuples $(x_i)\in U$ which are models of $\T$ form a Zariski closed subset of $U$ which is stable under the action of $\Ga=\GL_d(K)$. 
\end{lemma}
\begin{proof}
we can easily reduce to the case where $\T$ contains a single axiom $x\in Con^{p,q}$. Fix a basis $\{e_i\}$ of $W$ and a dual basis $\{f_i\}$ of $W^*$. 
If $Di$ is a diagram in $Con^{p,q}$ then $\Re^{p,q}(Di)$ can be written as
$$\sum_{j_1,\ldots,j_p,k_1,\ldots,k_q}f_{j_1,\ldots,j_p,k_1,\ldots,k_q}e_{j_1}\ot\cdots\ot e_{j_p}\ot f_{k_1}\ot\cdots\ot f_{k_q}$$ 
where $f_{j_1,\ldots,j_p,k_1,\ldots,k_q}$ are polynomials in $a^{\bullet}_{\bullet}$ where $a^{\bullet}_{\bullet}$ are defined in Subsection \ref{subsec:algstr} (see also the discussion on associativity of multiplication at the end of that subsection). 
In other words, every $r\in W^{q,p}$ gives a polynomial function $$f_r:=\langle\Re^{p,q}(x),r\rangle$$ which should vanish on models of $\T$, and the zero set of the polynomials $f_r$ is exactly the set of all models of $\T$ in $U$. This set is thus Zariski closed. It is easy to see that it is stable under $\Ga$, as a change of basis will not change the validity of axioms defined purely by diagrams.
\end{proof}
We write $J_{\T,d}= (f_r)_{r\in W^{q,p}}\subseteq K[U_d]$. We are thus interested in the Zariski closed set defined by $J_{\T,d}$. The coordinate ring for this closed set is $K[U_d]/rad(J_{\T,d})= K[Y_d]$. It thus holds that $I_{\T,d} = \Ker(K[U_d]\to K[Y_d])=rad(J_{\T,d}).$ It might be the case that $I_{\T,d}\neq J_{\T,d}$. This happens, for example, when we have an algebraic structure containing a single endomorphism $T:W\to W$. If we take the theory $\T = \{T^2\}$ then $\Tr(T)$ is in $I_{\T,d}$ but not in $J_{\T,d}$. Indeed, $\Tr(T)=0$ whenever $T^2=0$ so this polynomial will vanish on every model of the theory, and so it is in $I_{\T,d}$, but it cannot be contained in $J_{\T,d}$ since this ideal is generated in degree 2, and $\Tr(T)$ has degree 1.
The passage between the ideal and its radical will not play a huge difference here, since we have equality of ideals $rad(J_{\T,d}^{\Ga}) = (rad(J_{\T,d}))^{\Ga}$ in $K[U_d]^{\Ga}$. If $\dim_K W = d$ the ideal $J_{\T,d}^{\Ga}$ defines the subset of $X_d$ of all isomorphism classes with closed orbits of structures which are models for $\T$. 

We will describe now an ideal $\I_{\T}$ of $\winf$, and show that its image in $K[U_d]^{\Ga}$ is exactly $J_{\T,d}^{\Ga}$. We will use the fact that the ideal $J_{\T,d}$ is defined using pairing of the axioms with elements in the dual space $W^{q,p}$. We begin with the following definition:
\begin{definition}
We define a pairing $\pair^{p,q}:Con^{p,q}\ot Con^{q,p}\to Con^{0,0} = \winf$ on the basis elements by the following pictorial description:
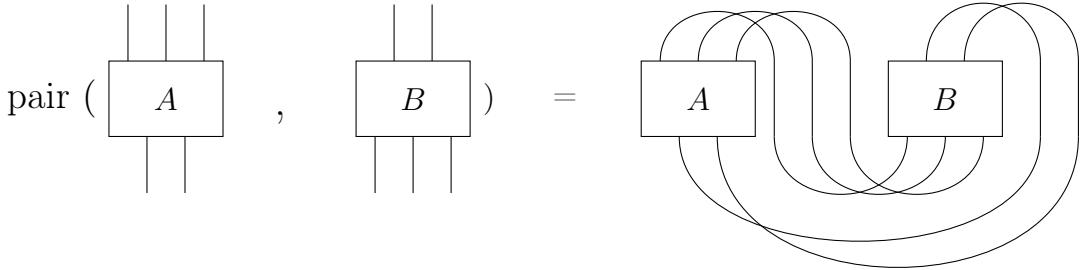
\begin{center}
\begin{tikzpicture}
	\begin{pgfonlayer}{nodelayer}
		\node [style=3function] (0) at (2, 0) {$A$};
		\node [style=3function] (1) at (5.25, 0) {$B$};
		\node [style=none] (2) at (1.75, -0.5) {};
		\node [style=none] (3) at (2.25, -0.5) {};
		\node [style=none] (4) at (1.5, 0.5) {};
		\node [style=none] (5) at (2, 0.5) {};
		\node [style=none] (6) at (2.5, 0.5) {};
		\node [style=none] (7) at (5, 0.5) {};
		\node [style=none] (8) at (5.5, 0.5) {};
		\node [style=none] (9) at (4.75, -0.5) {};
		\node [style=none] (10) at (5.25, -0.5) {};
		\node [style=none] (11) at (5.75, -0.5) {};
		\node [style=none] (12) at (3, 0.5) {};
		\node [style=none] (13) at (3.5, 0.5) {};
		\node [style=none] (14) at (4, 0.5) {};
		\node [style=none] (15) at (4, -0.5) {};
		\node [style=none] (16) at (3.5, -0.5) {};
		\node [style=none] (17) at (3, -0.5) {};
		\node [style=none] (18) at (7, 0.5) {};
		\node [style=none] (19) at (6.5, 0.5) {};
		\node [style=none] (20) at (7, -0.5) {};
		\node [style=none] (21) at (6.5, -0.5) {};
		\node [style=none] (22) at (0.25, 0) {=};
		\node [style=none] (23) at (-0.75, 0) {$\Huge ) $};
		\node [style=3function] (24) at (-5, 0) {$A$};
		\node [style=3function] (25) at (-1.75, 0) {$B$};
		\node [style=none] (26) at (-5.25, -0.5) {};
		\node [style=none] (27) at (-4.75, -0.5) {};
		\node [style=none] (28) at (-5.5, 0.5) {};
		\node [style=none] (29) at (-5, 0.5) {};
		\node [style=none] (30) at (-4.5, 0.5) {};
		\node [style=none] (31) at (-2, 0.5) {};
		\node [style=none] (32) at (-1.5, 0.5) {};
		\node [style=none] (33) at (-2.25, -0.5) {};
		\node [style=none] (34) at (-1.75, -0.5) {};
		\node [style=none] (35) at (-1.25, -0.5) {};
		\node [style=none] (39) at (-1.25, -1.25) {};
		\node [style=none] (40) at (-1.75, -1.25) {};
		\node [style=none] (41) at (-2.25, -1.25) {};
		\node [style=none] (42) at (-2, 1.25) {};
		\node [style=none] (43) at (-1.5, 1.25) {};
		\node [style=none] (44) at (-4.5, 1.25) {};
		\node [style=none] (45) at (-5, 1.25) {};
		\node [style=none] (46) at (-5.5, 1.25) {};
		\node [style=none] (47) at (-5.25, -1.25) {};
		\node [style=none] (48) at (-4.75, -1.25) {};
		\node [style=none] (49) at (-6.5, 0) {\large pair $($};
		\node [style=none] (50) at (-3.5, -0.25) {\Large ,};
	\end{pgfonlayer}
	\begin{pgfonlayer}{edgelayer}
		\draw [bend left=90, looseness=1.50] (4.center) to (12.center);
		\draw [bend left=270, looseness=1.50] (13.center) to (5.center);
		\draw [bend left=270, looseness=1.50] (14.center) to (6.center);
		\draw (12.center) to (17.center);
		\draw (13.center) to (16.center);
		\draw (14.center) to (15.center);
		\draw [bend right=90, looseness=1.50] (15.center) to (11.center);
		\draw [bend right=90, looseness=1.50] (16.center) to (10.center);
		\draw [bend right=90, looseness=1.50] (17.center) to (9.center);
		\draw [bend left=90, looseness=1.75] (8.center) to (18.center);
		\draw [bend left=90, looseness=1.75] (7.center) to (19.center);
		\draw (21.center) to (19.center);
		\draw (20.center) to (18.center);
		\draw [bend left=90] (21.center) to (2.center);
		\draw [bend left=90, looseness=1.25] (20.center) to (3.center);
		\draw (39.center) to (35.center);
		\draw (34.center) to (40.center);
		\draw (33.center) to (41.center);
		\draw (43.center) to (32.center);
		\draw (42.center) to (31.center);
		\draw (44.center) to (30.center);
		\draw (45.center) to (29.center);
		\draw (46.center) to (28.center);
		\draw (26.center) to (47.center);
		\draw (27.center) to (48.center);
	\end{pgfonlayer}
\end{tikzpicture}
\end{center}
\end{definition}
In other words- the pairing of $x\in Con^{p,q}$ with $y\in Con^{q,p}$
is given by connecting the $p$ free output strings of $x$ to the $p$ free input strings of $y$, and the $q$ free input strings of $x$ to the $q$ free output strings of $y$. We can give a rigorous formula using the formalism of Section \ref{sec:diagrams}, but it will not be very enlightening. 

Since connecting strings in diagrams corresponds to applying iteratively the evaluation $W\ot W^*\to K$ the following diagram is commutative:
\begin{equation}\xymatrix{ Con^{p,q}\ot Con^{q,p}\ar[rr]^{\Re^{p,q}\ot\Re^{q,p}}\ar[d]^{\pair^{p,q}} & & W^{p,q}\ot W^{q,p}\ar[d]^{ev} \\ \winf\ar[rr]^{\Re^{0,0}} & & K} \end{equation}
Here $\Re^{0,0}=\chi_{(W,(x_i))}:\winf\to K$ is the character of invariants of $W$ (see Definition \ref{def:charinv}).
The following lemma is immediate from the commutativity of the above diagram. 
\begin{lemma}
If $y\in Con^{q,p}$, $x\in \T\cap Con^{p,q}$ and $(W,(x_i))$ is a model of $\T$ then $\Re^{0,0}(\pair(x,y))=0$. \end{lemma}
\begin{definition}
We define $\I_{\T}\subseteq \winf$ to be the ideal generated by all elements of the form $\pair(x,y)$ where $x\in\T\cap Con^{p,q}$ and $y\in Con^{q,p}$, for some $(p,q)\in \N^2$.
\end{definition}

\begin{proposition}
We have $\Phi_d(\I_{\T}) = J_{\T,d}^{\Ga}$.
\end{proposition}
\begin{proof}(see also the proof of Theorem 1.4. in \cite{meir1})
Throughout the proof we will assume for simplicity that $\T$ contains a single axiom $x$ of degree $(p,q)$. The general case follows from the following argument:
If $\T = \{t_i\}_i$ and we write $\T_i$ for the singleton $\{t_i\}$ then it holds that $\I_{\T} = \sum_{i} \I_{\T_i}$ and $J_{\T,d} = \sum_i J_{\T_i,d}$. This means that we have a natural surjective map $$\oplus_i J_{\T_i,d}\to J_{\T,d}.$$ Since $\Ga$ is reductive this map remains surjective after taking invariants and we get a surjective map $\oplus_i J_{\T_i,d}^{\Ga}\to J_{\T,d}^{\Ga}$. In other words, it holds that $$J_{\T,d}^{\Ga} = \sum_i J_{\T_i,d}^{\Ga}$$ and it is enough to prove the proposition in case $\T$ is a singleton. 

The homomorphism $\Phi_d$ maps $\winf$ to $K[U_d]^{\Ga}$. We can also think of this as a homomorphism from $\winf$ to $K[U_d]$ with image in the $\Ga$-invariants. So to show inclusion in the first direction, it will be enough to show that $\Phi_d(\I_{\T})\subseteq J_{\T,d}$. 

To do so, we use the fact that $J_{\T,d}$ is generated by elements of the form $f_z:=\langle x,z\rangle$ where $z\in W^{q,p}$. Consider $\Phi_d(\pair(x,y))$. 
If we write $\{\epsilon_i\}$ for the basis of $W^{p,q}$ arising from the original basis of $W$ and $\{\delta_i\}$ for the dual basis of $W^{q,p}$ then $\Re^{p,q}(x)$ can be written as $\sum_i f_i\epsilon_i$ where $f_i\in K[U_d]$ and $y$ can similarly be written as $\sum_i g_i\delta_i$. It then follows that $\phi(\pair(x,y)) = \sum_i f_ig_i\in K[U_d]$ is contained in the ideal generated by $f_i$. But this is exactly the ideal $J_{\T,d}$.

For the other direction, consider the ideal $J_{\T,d}^{\Ga}$ 
Using the same terminology as above, a general element in $J_{\T,d}$ is of the form $\sum f_ih_i$ for some $h_i\in K[U_d]$. To phrase it in a different way, write $V= span\{f_i\}$. Then $V$ is a $\Ga$-representation, and we have a surjective map $V\ot K[U_d]\to J_{\T,d}$. The algebra $K[U_d]$ admits a direct sum decomposition given by the grading $$K[U] = \bigoplus_{n_1,\ldots,n_r}K[U]_{n_1,\ldots,n_r}.$$ Since $\Ga$ is reductive, taking $\Ga$-invariants preserves surjective maps. Taking $\Ga$-invariants also commute with direct sums so the map $$\zeta:\bigoplus_{n_1,\ldots,n_r}(V\ot K[U]_{n_1,\ldots,n_r})^{\Ga}\to I_{\T}^{\Ga}$$ is surjective. So it will be enough to prove that for every $n_1,\ldots, n_r$ the image of $(V\ot K[U]_{n_1,\ldots,n_r})^{\Ga}$ under $\zeta$ is contained in $\Phi_d(\I_{\T})$. Fix such a tuple $(n_1,\ldots,n_r)$ and write $V' = K[U]_{n_1,\ldots,n_r}$.

We have an isomorphism $V\cong (W^{p,q})^*=W^{q,p}$ as $\Ga$-representations, and $V'$ is a quotient of $W^{q',p'}$ where $p' = \sum_i n_ip_i$ and $q' = \sum_i n_iq_i$. If $p+p'\neq q+q'$ then $(V\ot W^{q',p'})^{\Ga}=0$ and as a result $(V\ot V')^{\Ga}=0$ and there is nothing to prove (again- we use here the reductivity of $\Ga$).
Assume that $p+p' = q+q'=n$. Schur-Weyl duality tells us that in this case all the resulting $\Ga$-invariants in $(V\ot V')$ arise from permutations $\sigma\in S_n$. In other words, following back all the identifications we have used, the resulting element of $J_{\T,d}^{\Ga}$ can be written as 
$$f=\Tr(L^{(n)}_{\sigma}\Re^{p,q}(x)\ot \xini).$$

This element can further be simplified to be of the form $\Phi_d(\pair(x,y))$ in the following way: write $Di = \xini$. Then the element $f$ is represented by the diagram
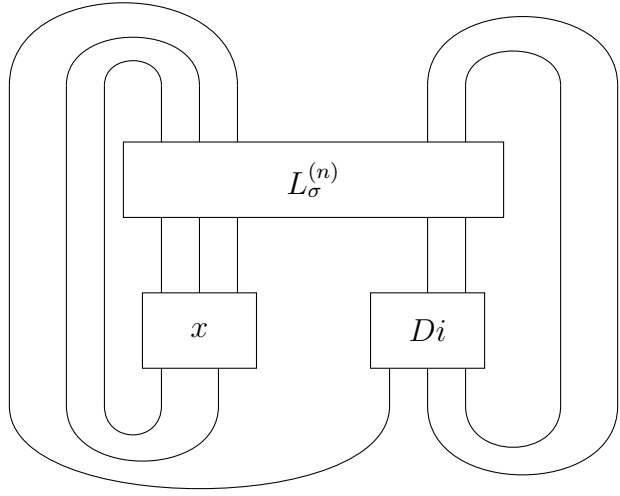
\begin{equation}
\begin{tikzpicture}
	\begin{pgfonlayer}{nodelayer}
		\node [style=none] (2) at (-3, 2) {};
		\node [style=none] (3) at (-3, 1) {};
		\node [style=none] (4) at (-3, 0) {};
		\node [style=none] (5) at (-3, -0.5) {};
		\node [style=3function] (6) at (-2.5, 0.5) {$x$};
		\node [style=3function] (7) at (0.5, 0.5) {$Di$};
		\node [style=none] (8) at (-2.5, 2) {};
		\node [style=none] (9) at (-2.5, 1) {};
		\node [style=none] (10) at (-2.25, 0) {};
		\node [style=none] (11) at (-2.25, -0.5) {};
		\node [style=none] (12) at (0.5, 2) {};
		\node [style=none] (13) at (0.5, 1) {};
		\node [style=none] (14) at (0.5, 0) {};
		\node [style=none] (15) at (0.5, -0.5) {};
		\node [style=none] (16) at (1, 2) {};
		\node [style=none] (17) at (1, 1) {};
		\node [style=none] (18) at (1, 0) {};
		\node [style=none] (19) at (1, -0.5) {};
		\node [style=none] (20) at (-2, 2) {};
		\node [style=none] (21) at (-2, 1) {};
		\node [style=none] (22) at (0, 0) {};
		\node [style=none] (23) at (0, -0.5) {};
		\node [style=multi function] (24) at (-1, 2.5) {$L^{(n)}_{\sigma}$};
		\node [style=none] (25) at (-3, 3) {};
		\node [style=none] (26) at (-2.5, 3) {};
		\node [style=none] (27) at (0.5, 3) {};
		\node [style=none] (28) at (1, 3) {};
		\node [style=none] (29) at (-2, 3) {};
		\node [style=none] (30) at (-3, 3.75) {};
		\node [style=none] (31) at (-2.5, 3.75) {};
		\node [style=none] (32) at (0.5, 3.75) {};
		\node [style=none] (33) at (1, 3.75) {};
		\node [style=none] (34) at (-2, 3.75) {};
		\node [style=none] (35) at (-5, 3.75) {};
		\node [style=none] (36) at (-4.25, 3.75) {};
		\node [style=none] (37) at (-3.75, 3.75) {};
		\node [style=none] (39) at (-4.25, -0.5) {};
		\node [style=none] (40) at (-3.75, -0.5) {};
		\node [style=none] (41) at (-5, -0.5) {};
		\node [style=none] (42) at (2.25, 3.75) {};
		\node [style=none] (43) at (3, 3.75) {};
		\node [style=none] (44) at (2.25, -0.5) {};
		\node [style=none] (45) at (3, -0.5) {};
	\end{pgfonlayer}
	\begin{pgfonlayer}{edgelayer}
		\draw (4.center) to (5.center);
		\draw (10.center) to (11.center);
		\draw (22.center) to (23.center);
		\draw (14.center) to (15.center);
		\draw (18.center) to (19.center);
		\draw (16.center) to (17.center);
		\draw (12.center) to (13.center);
		\draw (8.center) to (9.center);
		\draw (20.center) to (21.center);
		\draw (2.center) to (3.center);
		\draw (30.center) to (25.center);
		\draw (31.center) to (26.center);
		\draw (34.center) to (29.center);
		\draw (32.center) to (27.center);
		\draw (33.center) to (28.center);
		\draw (42.center) to (44.center);
		\draw (43.center) to (45.center);
		\draw (37.center) to (40.center);
		\draw (39.center) to (36.center);
		\draw (35.center) to (41.center);
		\draw [bend left=90, looseness=1.75] (5.center) to (40.center);
		\draw [bend left=270, looseness=1.50] (30.center) to (37.center);
		\draw [bend left=270, looseness=1.25] (31.center) to (36.center);
		\draw [bend left=90, looseness=1.25] (11.center) to (39.center);
		\draw [bend left=90, looseness=1.25] (35.center) to (34.center);
		\draw [bend right=90, looseness=0.75] (41.center) to (23.center);
		\draw [bend left=90, looseness=1.25] (33.center) to (42.center);
		\draw [bend left=270, looseness=1.25] (43.center) to (32.center);
		\draw [bend right=90, looseness=1.50] (19.center) to (44.center);
		\draw [bend left=90, looseness=1.25] (45.center) to (15.center);
	\end{pgfonlayer}
\end{tikzpicture}
\end{equation}
If all the output strings of $x$, after being shuffled by $L_{\sigma}^{(n)}$ are connected to output strings of $Di$, then what we get here is a pairing of $x$ with the diagram resulted from $Di$ by closing some of its input strings with some of its output strings. If, on the other hand, some output strings of $x$ are connected to input strings of $x$ after shuffling,  we can use Equation \ref{eq:usingid} and add some copies of $\Id_W$ to $Di$ to break down these strings into two- the output string from $x$ will go to the input string of this $\Id_W$ box, and the output string of this box will go to the relevant input string of $x$. By taking the diagram formed from $Di\star (\Id_W)^{\star m}$ and closing it accordingly, we still get that the invariant is of the form $\pair(x,y)$ for some $y\in Con^{q,p}$ and we are done.
\end{proof}

\section{PSH-algebra structure of $\Ainf$ and a proof of Theorem \ref{thm:Ainfpsh}}\label{sec:pshalg}
We recall the following definition from \cite{Zelevinsky}:
\begin{definition}
A positive self-adjoint Hopf algebra (or PSH-algebra) is an $\N$-graded $\Z$-Hopf algebra $A$ equipped with a graded basis $B$ of $A$ and a pairing $\langle-,-\rangle:A\ot_{\Z} A\to \Z$ such that the following conditions are satisfied:
\begin{enumerate}
\item The basis $B$ is orthonormal with respect to $\langle-,-\rangle$. In other words- for every $x,y\in B$ we have $\langle x,y\rangle = \delta_{x,y}$.
\item The multiplication is adjoint to the comultiplication with respect to $\langle -,- \rangle$ where $A\ot_{\Z}A$ has the tensor product pairing. 
\item The unit $u:\Z\to A$ and the counit $\epsilon:A\to\Z$ are adjoint with respect to $\langle -,-\rangle$ where $\Z$ has the canonical pairing.
\item The algebra $A$ is connected, that is $A_0=\Z$.
\item All the structure constants of $m,\Delta,u,\epsilon$ with respect to the basis $B$ are non-negative integers.
\end{enumerate}
By a graded basis we mean that $B=\sqcup_n B_n$, where $B_n$ is a basis for $A_n$. 
\end{definition}
Our algebra $\Ainf$ is not defined over $\Z$ but over a field $K$ of characteristic zero. Nevertheless, all the structure constants for $m$ and $\Delta$ with respect to the basis $B$ of monomials in the irreducible basic invariants are non-negative integers. We now adjust the definition of Zelevinsky to fit in our framework.
\begin{definition}
A rational $K$-PSH-algebra is an $\N^r$-graded $K$-Hopf algebra $A$ equipped with a graded basis $B$ which satisfies the following conditions:
\begin{enumerate}
\item The basis $B$ is orthogonal and positive with respect to $\langle-,-\rangle$. In other words- for every $x,y\in B$ we have $\langle x,y\rangle = \delta_{x,y}c(x)$ for some $c(x)\in\Q_{+}.$
\item The multiplication is adjoint to the comultiplication with respect to $\langle -,- \rangle$ where $A\ot_K A$ has the tensor product pairing. 
\item The unit $u:K\to A$ and the counit $\epsilon:A\to K$ are adjoint with respect to $\langle -,-\rangle$ where $K$ has the canonical pairing.
\item The algebra $A$ is connected, that is $A_0=K$.
\item all the structure constants of $m,\Delta,u,\epsilon$ with respect to the basis $B$ are in $\Q_{+}$.
\end{enumerate}
The number $r$ which appears in the grading is some positive integer. 
\end{definition}
In the sequel we will refer simply to rational PSH-algebra when the field $K$ will be clear from context. 
The serious relaxation we made was in the first axiom, where instead of talking about orthonormal basis we speak now about orthogonal basis. It turns out that the basis of monomials we have for $\Ainf$ will furnish a natural structure of a rational PSH-algebra, but for which the constants $c(x)$ will often be integers $\neq 1$. 

We define the following inner product on $(\Ainf)_{n_1,\ldots,n_r}$ for every tuple $(n_1,\ldots,n_r)\in \N^r$:

$$\langle p(n,\sigma,n_1,\ldots,n_r),p(n,\tau,n_1,\ldots n_r)\rangle = |\{g\in S_{n_1,\ldots n_r}| \alpha_1(g)\sigma\alpha_2(g^{-1}) = \tau\}|$$
where we write $\alpha_1 = \alpha^{(q_i)}_{(n_i)}$ and $\alpha_2 = \alpha^{(p_i)}_{(n_i)}$. 
Since equality between basic invariants is defined using the action of $S_{n_1,\ldots, n_r}$ it holds that if $p(n,\sigma,n_1,\ldots,n_r)\neq p(n,\tau,n_1,\ldots,n_r)$ then their pairing is zero, and the squared norm of $p(n,\sigma,n_1,\ldots,n_r)$ is the cardinality of its stabilizer in $S_{n_1,\ldots,n_r}$, where the action of this group is given by $(g,\sigma)\mapsto \alpha_1(g)\sigma\alpha_2(g^{-1})$. We will think of this stabilizer as the automorphism group of the diagram of $p(n,\sigma,n_1,\ldots,n_r)$, as it corresponds to permutations of the boxes in a way which does not change the diagram. 
This pairing can also be interpreted in the following way:
we have natural identification and inclusion 
$$(\Ainf)_{n_1,\ldots,n_r}\cong K[S_n]_{S_{n_1,\ldots,n_r}}\cong K[S_n]^{S_{n_1,\ldots, n_r}}\subseteq K[S_n]$$ and $K[S_n]$ has a natural inner product given by $\langle \sigma,\tau\rangle = \delta_{\sigma,\tau}$. The inner product described here is just the restriction of this inner product, rescaled by $\prod_i (n_i)!$ (see also Lemma \ref{lem:compinnerproducts}).
The rest of this section will be devoted to proving the following:
\begin{theorem}
The algebra $\Ainf$ with the basis of basic invariants and the above inner product is a rational PSH-algebra.
\end{theorem}
Since all basic invariants can be written uniquely as monomials in the irreducible basic invariants, and since the irreducible basic invariants are primitive with respect to $\Delta$, it is easy to see that most of the axioms of a rational PSH-algebra hold in $\Ainf$. The only non-trivial part is the fact that $m$ is adjoint to $\Delta$ with respect to $\langle -,-\rangle$. The proof of the theorem will thus be complete with the proof of the following lemma:
\begin{lemma}
For every $a,b,c\in \Ainf$ it holds that $$\langle a\ot b, \Delta(c)\rangle = \langle ab,c\rangle.$$
\end{lemma}
\begin{proof}
Due to linearity it is enough to prove this in case $a,b$ and $c$ are monomials in the irreducible basic invariants. Write $$a = p_1^{a_1}\cdots p_t^{a_t}, b=p_1^{b_1}\cdots p_t^{b_t}\text{ and } c= p_1^{c_1}\cdots p_t^{c_t}$$
where $p_1,\ldots,p_t$ are irreducible basic invariants. 
Since the monomials form an orthonormal basis, we get that if $a_i+b_i\neq c_i$ for some $i$ then both sides of the equation are zero. So we assume that $a_i+b_i=c_i$ for $i=1\ldots,t$. 
It follows that $ab=c$ and the right hand side is equal to $\langle c,c\rangle = |\Aut(c)|$. On the other hand, since all the elements $p_i$ are primitive, we get that $\Delta(p_i^{c_i}) = \sum _{r_i+s_i=c_i}\binom{c_i}{r_i}p_i^{r_i}\ot p_i^{s_i}$
It follows that 
$$\langle a\ot b, \Delta(c)\rangle = \sum_{\substack{r_1+s_1=c_1\\ \vdots \\ r_t+s_t=c_t}}\langle p_1^{a_1}\cdots p_t^{a_t}\ot p_1^{b_1}\cdots p_t^{b_t},\binom{c_1}{r_1}\cdots\binom{c_t}{r_t}p_1^{r_1}\cdots p_t^{r_t}\ot p_1^{s_1}\cdots p_t^{s_t}\rangle=$$
$$\binom{c_1}{a_1}\cdots \binom{c_t}{a_t}\langle a,a\rangle\langle b,b\rangle.$$
We thus need to prove the equation $$\langle c,c\rangle = \prod_i \binom{c_i}{a_i} |\Aut(a)||\Aut(b)|.$$
This will follow once we prove that 
$$|\Aut(a)| = \prod_i (a_i!|\Aut(p_i)|^{a_i}).$$
This follows easily from the following argument: every automorphism of $a$ permutes the connected components of the diagram which corresponds to $a$. The diagram of $a$ has $a_1+a_2+\ldots + a_t$ connected components which correspond to the $p_i$ constituents. Since an automorphism of $a$ will send a connected component to an equivalent connected component any automorphism of $a$ will permute the connected components of type $p_1$, the connected components of type $p_2$ and so on.

We thus have a surjective group homomorphism $\Aut(a)\to \prod_i S_{a_i}$. 
The kernel of this homomorphism consists of those homomorphisms of $a$ which fix all the connected components, and is thus isomorphic to $\prod_i \Aut(p_i)^{a_i}$. 
Calculating the cardinality of $\Aut(a)$ now gives us the desired result. 
\end{proof}
So we do see that we get a rational PSH-algebra structure on $\Ainf$.
PSH-algebras play an important role in the representation theory of finite groups. 
In \cite{Zelevinsky} Zelevinsky proved that every PSH-algebra decomposes uniquely into the tensor product of what he called \emph{universal} PSH-algebras. 
The universal PSH-algebra is the polynomial algebra $\Z[x_1,x_2,\ldots]$ where $\deg(x_n)=n$ and where $\Delta(x_n) = \sum_{a+b=n} x_a\ot x_b$. 
This algebra arises in the study of the representation theory of the symmetric groups in the following way:
Define $A= \bigoplus_{n\geq 0}\R(S_n)$, the direct sum of the Grothendieck groups of all the symmetric groups. For $[V]\in \R(S_n)$ and $[W]\in \R(S_m)$ define $$[V]\cdot [W] = [\Ind_{S_n\times S_m}^{S_{n+m}}V\ot W]$$
$$\Delta([V]) = \sum_{a+b=n} \Res^{S_n}_{S_a\times S_b}[V]\in \bigoplus_{a+b=n}\R(S_a)\ot \R(S_b).$$
These operations together with the basis given by the irreducible representations of $S_n$ define a structure of a PSH-algebra. In \cite{Zelevinsky} Zelevinsky described other PSH-algebras arising from other families of finite groups such as wreath products or general linear groups over finite fields.
In Section \ref{sec:example} we will show that when our algebraic structure contains a single endomorphism the algebra $\Ainf$ is just the extension of scalars of the universal PSH-algebra of Zelevinsky. I do not know, however, if the rational PSH-algebra $\Ainf$ always has a representation theoretic interpretation. See also Question \ref{qu:repth}. 

\section{The Hilbert function of $\Ainf$ and of its finitely generated quotients}\label{sec:hilbert}

We explain now how to calculate the Hilbert function of $\Ainf$ and also of the finitely generated quotients $\Ainf/I_d$. 
The algebra $\Ainf$ is graded by $\N^r$, and we define the Hilbert function $H$ of $\Ainf$ to be 
$$H(n_1,\ldots,n_r) = \dim_K (\Ainf)_{n_1,\ldots,n_r}.$$
Our goal is to prove the following theorem:
\begin{theorem}
The function $H(n_1,\ldots, n_r)$ can be expressed explicitly using the Littlewood-Richardson coefficients and the Kronekcer coefficients in the following way:
If $\sum_i n_ip_i\neq \sum_i n_iq_i$ then it is 0, and if $n=\sum_i n_ip_i = \sum_i n_iq_i$ then it is equal to 
$$\sum_{\la\parti n,\underline{\mu},\underline{\nu},\rho_1,\ldots \rho_r} c^{\la}_{\underline{\mu}}g(\mu_{11},\mu_{12},\ldots,\mu_{1p_1},\rho_1)g(\mu_{21},\mu_{22},\ldots,\mu_{2p_2},\rho_2)\cdots g(\mu_{r1},\ldots,\mu_{rp_r},\rho_r)\cdot $$
$$
c^{\la}_{\underline{\nu}}g(\nu_{11},\nu_{12},\ldots,\nu_{1q_1},\rho_1)g(\nu_{21},\nu_{22},\ldots,\nu_{2q_2},\rho_2)\cdots g(\nu_{r1},\ldots,\nu_{rq_r},\rho_r)$$
where the sum is taken over all $$\underline{\mu} = (\mu_{ij})_{i=1,\ldots r, j=1,\ldots p_i} ,
\underline{\nu} = (\nu_{ij})_{i=1,\ldots r, j=1,\ldots q_i}, 
\underline{\rho} = (\rho_i)_{i=1,\ldots r}$$
$$ \mu_{ij},\nu_{ij},\rho_i\parti n_i.$$
The Hilbert function of $\Ainf/I_d$, is similar, except that the sum is taken only over the partitions $\la\parti n$ such that the number of rows in $\la$ is $\leq d$.
\end{theorem}
 \begin{proof}
If $\sum_i n_i(p_i-q_i)\neq 0$ then $(\Ainf)_{n_1,\ldots n_r}=0$. 
If $\sum_i n_ip_i = \sum_i n_iq_i = n$ then
$(\Ainf)_{n_1,\ldots n_r}$ is spanned by the elements $p(n,\sigma,n_1,\ldots,n_r)$, where for $(\sigma_i)\in S_{n_1,\ldots n_r}$ we have that 
$$p(n,\sigma,n_1,\ldots,n_r) = p(n,\alpha^{(q_i)}_{(n_i)}((\sigma_i))\sigma\alpha^{(p_i)}_{(n_i)}((\sigma_i^{-1})),n_1,\ldots, n_r).$$
In other words, we have a surjective map
$$KS_n\to (\Ainf)_{n_1,\ldots n_r}$$ which splits via
$$KS_n\to (KS_n)_{S_{n_1,\ldots,n_r}}\stackrel{\cong}{\to} (\Ainf)_{n_1,\ldots n_r}.$$ Here the action of $S_{n_1,\ldots, n_r}$ on $KS_n$ is given by the formula 
$$(\sigma_i)\cdot \sigma = \alpha^{(q_i)}_{(n_i)}((\sigma_i))\sigma\alpha^{(p_i)}_{(n_i)}((\sigma_i^{-1})).$$
For short, we shall write here, as in the previous section, $\alpha_1 = \alpha^{(q_i)}_{(n_i)}$ and $\alpha_2 = \alpha^{(p_i)}_{(n_i)}$. 
We thus need to calculate the dimension of $(KS_n)_{S_{n_1,\ldots,n_r}}$.
Since $S_{n_1,\ldots,n_r}$ is a finite group it holds that 
$$(KS_n)_{S_{n_1,\ldots,n_r}}\cong (KS_n)^{S_{n_1,\ldots,n_r}}.$$
Now, using Wedderburn decomposition we have 
$$KS_n\cong \bigoplus_{\la\parti n}\Hom_K(\S_{\la},\S_{\la}).$$
Thus,
$$(KS_n)^{S_{n_1,\ldots,n_r}}\cong \bigoplus_{\la\parti n}(\Hom_K(\S_{\la},S_{\la}))^{S_{n_1,\ldots, n_r}}\cong 
\bigoplus_{\la\parti n} \Hom_{S_{n_1,\ldots,n_r}}(\S_{\la},\S_{\la}),
$$ where the action of $S_{n_1,\ldots,n_r}$ on one copy of $\S_{\la}$ is by $\alpha_1$, and on the other copy of $\S_{\la}$ by $\alpha_2$. 
In other words, we get a linear isomorphism 
$$(\Ainf)_{n_1,\ldots n_r}\cong \bigoplus_{\la\parti n} \Hom_{S_{n_1,\ldots,n_r}}(\alpha_2^*(\S_{\la}),\alpha_1^*(\S_{\la})).$$
To calculate the dimension of this space we will first calculate $\alpha_i^*(\S_{\la})$ for a partition $\la\parti n$ and $i=1,2$. 

The group homomorphism $\alpha_1$ is given by the following composition:
$$S_{n_1,\ldots n_r}=S_{n_1}\times\cdots\times S_{n_r}
\stackrel{\Omega^{(q_1^{n_1})}\times\cdots\times\Omega^{(q_r^{n_r})}}{\longrightarrow} $$
$$S_{n_1q_1}\times\cdots\times S_{n_rq_r}\stackrel{\Pi_{(n_1q_1,\ldots,n_rq_r)}}{\to} S_n$$
Using Remark \ref{rem:uptoconj} we can rewrite this map, up to conjugation, as  
$$S_{n_1}\times\cdots S_{n_r}\to (S_{n_1})^{q_1}\times\cdots\times (S_{n_r})^{q_r}\stackrel{\Pi_{(n_1^{q_1})}\times\cdots\times \Pi_{(n_r^{q_r})}}{\longrightarrow} $$
$$S_{n_1q_1}\times\cdots\times S_{n_rq_r}\stackrel{\Pi_{(n_1q_1,\ldots, n_rq_r)}}{\to} S_n$$ where the first map is given by the product of the diagonal embeddings $S_{n_i}\to S_{n_i}^{q_i}$. We can further rewrite this map as
$$\ol{\alpha_1}:S_{n_1}\times\cdots\times S_{n_r}\to (S_{n_1})^{q_1}\times\cdots\times (S_{n_r})^{q_r}\stackrel{\Pi_{(n_1^{q_1},\ldots, n_r^{q_r})}}{\to} S_n.$$
Since conjugation does not change the isomorphism type of representations we have $$\alpha_1^*(\S_{\la})\cong (\ol{\alpha_1})^*(\S_{\la}).$$
Let us write $diag:S_{n_1}\times\cdots\times S_{n_r}\to (S_{n_1})^{q_1}\times\cdots\times (S_{n_r})^{q_r}$ for the product of the diagonal embeddings. We then have
$$\alpha_1^*[\S_{\la}] = diag^*\Pi_{(n_1^{q_1},\ldots,n_r^{q_r})}^*[\S_{\la}] = $$
$$\sum_{\nu_{i,j}\parti n_i}diag^*(c^{\la}_{(\nu_{ij})}[\S_{\nu_{11}}\boxtimes\cdots\boxtimes\S_{\nu_{r,q_r}}]) = $$
$$\sum_{\nu_{i,j}\parti n_i}\sum_{\rho_i\parti n_i}c^{\la}_{(\nu_{ij})}g(\nu_{11},\nu_{12},\ldots,\nu_{1q_1},\rho_1)\cdots g(\nu_{r1},\nu_{r2},\ldots\nu_{rq_r},\rho_r)[\S_{\rho_1}\boxtimes\cdots\boxtimes \S_{\rho_r}]$$
where the sum is taken over all tuples of partitions $(\nu_{i,j})$, $i=1,\ldots ,r$, $j=1,\ldots q_i$ and $\rho_i\parti n_i$ for $i=1,\ldots, r$. We used here Definition \ref{def:LRco} for the iterated Littlewood-Richardson coefficients and Equation \ref{eq:Kronecker} for the iterated Kronecker coefficients. 

Applying a similar calculation for $\alpha_2$ gives us 
$$\alpha_2^*[\S_{\la}] = \sum_{\mu_{i,j}\parti n_i}\sum_{\rho_i\parti n_i}c^{\la}_{(\mu_{ij})}g(\mu_{11},\mu_{12},\ldots,\mu_{1p_1},\rho_1)\cdots g(\mu_{r1},\mu_{r2},\ldots\mu_{rp_r},\rho_r)[\S_{\rho_1}\boxtimes\cdots\boxtimes \S_{\rho_r}]$$
Since $$\dim\Hom_{S_{n_1}\times\cdots\times S_{n_r}}(\S_{\rho_1}\boxtimes\cdots\boxtimes \S_{\rho_r}, \S_{\rho'_1}\boxtimes\cdots\boxtimes \S_{\rho'_r})= \delta_{\rho_1,\rho'_1}\cdots \delta_{\rho_r,\rho'_r}$$ 
we get the formula in the theorem. 
The claim about the Hilbert function of $\Ainf/I_d$ is derived in a similar way, using the fact that if we divide by the ideal $I_d$, the isomorphism 
$$\bigoplus_{\la\parti n}\Hom_{S_{n_1}\times\cdots\times S_{n_r}}(\S_{\la},S_{\la})\to (\Ainf)_{n_1,\ldots,n_r}$$ reduces to an isomorphism 
$$\bigoplus_{\substack{\la\parti n\\ r(\la)\leq d}}\Hom_{S_{n_1}\times\cdots\times S_{n_r}}(\S_{\la},\S_{\la})\to (\Ainf/I_d)_{n_1,\ldots,n_r}.$$ 
\end{proof}

\section{Case study- algebraic structure of a vector space with a single endomorphism}\label{sec:example}
This section contains a detailed case study of the algebra $\Ainf$ in case the algebraic structure we have is that of a vector space with a single linear endomorphism.
In other words: there is a unique tensor structure $x_1=T$ of degree $(1,1)$.
The algebra $\Ainf$ is then graded by $\N^r= \N^1=\N$, where $(\Ainf)_n$ is spanned by elements of the form $p(n,\sigma,n)$.
The embeddings $\alpha_1,\alpha_2:S_n\to S_n$ are then just the identity maps, and $$p(n,\sigma,n) = \Tr(L^{(n)}_{\sigma}T^{\ot n}).$$ 
In this case $p(n,\sigma,n)=p(n,\tau,n)$ if and only if $\sigma$ and $\tau$ are conjugate in $S_n$. 
If the cycle lengths of $\sigma$ are $\la_1,\ldots \la_t$ then it holds that $$p(n,\sigma,n) = \Tr(T^{\la_1})\cdots\Tr(T^{\la_t}).$$
The formula for multiplication of basic invariants gives us 
$$p(n,\sigma,n)\cdot p(m,\tau,m) = p(n+m,(\sigma,\tau),n+m)$$ where $(\sigma,\tau)\in S_{n+m}$ by the natural embedding $\Pi_{(n,m)}:S_n\times S_m\to  S_{n+m}$. 

The irreducible basic invariants are thus the permutations which consist of a single cycle. 
Up to conjugation, we can represent such irreducible basic invariants as $\Tr(T^n) = p((123\cdots n),n)$. These are primitive elements with respect to $\Delta$ and group like elements with respect to $\Delot$. 
The pairing on $(\Ainf)_n\cong (KS_n)_{S_n}$ is then given by 
$$\langle \sigma,\tau \rangle  = \begin{cases} |C_{S_n}(\sigma)| \text{ if }\sigma \text{ and } \tau \text{ are conjugate} \\ 0 \text{ if } \sigma \text{ and } \tau \text{ are not conjugate} \end{cases} $$
So for example the squared norm of $p((123\cdots n),n)$ is $n$, because the centralizer of $(123\cdots n)$ is the cyclic group generated by $(123\cdots n)$, which has order $n$. 

We will now show that this algebra has a different basis, parametrized by partitions of $n$ for every $n\in\N$. 
This will enable is to show that this algebra has a $\Z$-lattice which is the universal PSH-algebra of Zelevinsky, and it will enable us to give a clean description of the ideals $I_d$. 
For this, recall that since $S_n$ is a finite group the natural map
$$\phi:(KS_n)^{S_n}\to (KS_n)_{S_n}$$ $$z\mapsto \ol{z}$$ is an isomorphism with inverse given by $$\phi^{-1}(\ol{\sigma}) = \frac{1}{n!}\sum_{\tau\in S_n} \tau\sigma\tau^{-1}.$$
Also, since the action of $S_n$ on $KS_n$ here is simply conjugation, the invariant subspace $(KS_n)^{S_n}$ is just the center of $S_n$. The center of $S_n$ has a natural basis given by $\{e_{\la}\}_{\la\parti n}$ where $e_{\la}$ is the central idempotent which corresponds to the Specht module $\S_{\la}$

We calculate here the squared norm of $e_{\la}$. For this, define first an inner product on $KS_n$ by $$\langle \sigma,\tau \rangle = \begin{cases} 0 \text{ if }\sigma\neq \tau \\ n! \text{ if } \sigma=\tau.\end{cases}$$ This inner product is the same as $\langle\sigma,\tau\rangle = \chi_{reg}(\sigma\tau^{-1})$ where $\chi_{reg}$ is the character of the regular representation of $S_n$. 
We claim the following:
\begin{lemma}\label{lem:compinnerproducts} The map $\Ainf_n\cong (KS_n)_{S_n}\to (KS_n)^{S_n}\to KS_n$ preserves the inner product. 
\end{lemma}
\begin{remark}
As mentioned in the definition of the inner product in Section \ref{sec:pshalg}, this lemma is the reason we have defined the inner product on the coinavriants space the way we did.
\end{remark}
\begin{proof}
The image of $\ol{\sigma}\in K[X]_n$ in $(KS_n)^{S_n}$ is $\frac{1}{n!}\sum_{x\in S_n} x\sigma x^{-1}$. We calculate:
$$\langle \phi^{-1}(\ol{\sigma}),\phi^{-1}(\ol{\tau})\rangle = \frac{1}{n!^2}\sum_{x,y\in S_n}\langle x\sigma x^{-1}, y\tau y^{-1}\rangle = $$
$$\frac{1}{n!}\sum_{x\in S_n} \langle x\sigma x^{-1},\tau\rangle  = \begin{cases} 0 \text{ if } \sigma\text{ and } \tau \text{ are not conjugate} \\
|C_{S_n}(\sigma)| \text{ if } \sigma \text{ and } \tau \text{ are conjugate.} \end{cases}$$
But this is the same as $\langle \ol{\sigma},\ol{\tau}\rangle$ so we are done.
\end{proof}
Now we can calculate the squared norm of $\phi(e_{\la})$. Since every element of $S_n$ is conjugate to its inverse, and since $e_{\la}$ is an idempotent, we get that 
$$\langle \phi(e_{\la}),\phi(e_{\la})\rangle = \chi_{reg}(e_{\la}) = d_{\la}^2$$ where $d_{\la}$ is the dimension of the Specht module $\S_{\la}$.
\begin{definition}
We write $\laa = \frac{1}{d_{\la}}\phi(e_{\la})$.
\end{definition}
We use this terminology to be consistent with Zelevinsky (see Chapter II.6 in \cite{Zelevinsky}).

We thus have a new basis for $\Ainf$, given by $\big\{\laa\big\}_{\la\parti n, n\in \N}$. This basis is orthonormal. Let $A\subseteq \Ainf$ be the $\Z$ lattice generated by the elements $\laa$. We can already identify between $A$ and $\bigoplus_{n\geq 0} \R(S_n)$, since both of them have a basis parametrized by partitions of all natural numbers. In addition, the basis here is also orthonormal, as in the universal PSH-algebra. Our goal now will be to show that the algebra $A$ has the same multiplication and comultiplication as the algebra of Zelevinsky. 

It will be enough to show that the product is the same, since we know that in both algebras the comulutplication is adjoint to the multiplication, and so the multiplication defines it uniquely. 

So let $\laa\in A_n$ and $\muu\in A_m$. The product of $\laa$ and $\muu$ is just given by considering them as elements in $(KS_n)_{S_n}$ and $(KS_m)_{S_m}$ and considering the image of their tensor product under the map $$(KS_n)_{S_n}\ot (KS_m)_{S_m}\to K(S_n\times S_m)_{S_n\times S_m}\to (KS_{n+m})_{S_{n+m}}.$$
By applying $\phi^{-1}$ again we can write 
$$\phi^{-1}(\laa\muu) = \frac{1}{d_{\la}d_{\mu}(n+m)!}\sum_{x\in S_{n+m}}x\Pi_{(n,m)}(e_{\la},e_{\mu})x^{-1} = $$
$$\sum_{\nu\parti n+m} a_{\nu} \frac{1}{d_{\nu}}e_{\nu}$$
where $a_{\nu}\in \Q$ are some numbers. 
The following lemma proves that these scalars are exactly the Littlewood-Richardson coefficients. 
\begin{lemma}
Let $G$ be a finite group and let $H$ be a subgroup. Assume that $e\in H$ is the idempotent which corresponds to an irreducible representation $V$ of $H$.
Write $W_1,\ldots W_r$ for the isomorphism classes of the irreducible representations of $G$, and let $f_i$ be their corresponding central idempotents. Then 
$$\frac{1}{|G|\dim(V)}\sum_{x\in G} xex^{-1} = \sum_i \frac{a_i}{\dim(W_i)}f_i$$ where $a_i = \dim\Hom_G(\Ind_H^G(V),W_i)= \dim\Hom_H(V,\Res^G_H(W_i))$. 
\end{lemma}
\begin{proof}
Write $Z = \frac{1}{|G|\dim(V)}\sum_{x\in G} xex^{-1}= \sum_i \frac{a_i}{\dim(W_i)}f_i$ and 
write $\chi_i$ for the character of $W_i$. 
Since $\chi_i(f_j) = \dim(W_i)\delta_{i,j}$ we get 
$\chi_i(Z) = a_i$, so it will be enough to evaluate this expression.
Using the fact that $\chi_i$ is invariant under conjugation in $G$ we get that 
$$a_i=\chi_i(Z) = \frac{1}{\dim(V)}\chi_i(e).$$ 
Now if $\dim\Hom_H(V,\Res^G_H(W_i))=t_i$ then $\Res^G_H(W_i) \cong  V^{\oplus t_i}\oplus V'$ where $V'$ is a direct sum of irreducible $H$-representations which are not isomorphic to $V$. The action of $e$ on $W_i$ is then given by projection on $V^{\oplus t_i}$ with kernel $V'$. But since this is a projection with image of rank $\dim(V) t_i$, we get that $a_i=\frac{1}{\dim(V)}\dim(V)t_i = t_i$. and we are done. 
\end{proof}

This shows that under the identification $A_n\cong \R(S_n)$, $\{\la\}\mapsto [\S_{\la}]$ the multiplication is given by $$[V]\cdot[W] = [\Ind_{S_n\times S_m}^{S_{n+m}}V\ot W]$$ which is the same as the multiplication in the algebra of Zelevinsky.
Write $X_n\in A_n$ for the class of the trivial representation of $S_n$. Zelevinsky showed that $A= \Z[X_1,X_2,\ldots]$. We Can thus write $\Ainf\cong K[\Tr(T),\Tr(T^2),\ldots]\cong K[X_1,X_2,\ldots]$. The two resulting sets of monomials for $\Ainf$ are quite different here. 
The $\Z$-basis for $A$ given by monomials in $X_i$ is useful in determining the ideals $I_d$. Indeed, from Schur-Weyl duality we know that $I_d$ is spanned exactly by those $\laa$ where the partition $\la$ has more than $d$ rows. The algebra $A$ has a unique non-trivial PSH-algebra automorphism $t$ (see Section 3.11 and 4.1 in \cite{Zelevinsky}) which sends $\laa$ to $\{\la^t\}$ where $\la^t$ is the partition of $n$ given by taking the transpose of the Young diagram of $\la$. This implies that the ideal $t(I_d)$ is generated by all $\laa$ where $\la$ has more than $d$ columns. We claim the following:
\begin{lemma}
the ideal $t(I_d)$ is equal to $(X_{d+1},X_{d+2},\ldots)$. 
As a result $\Ainf/I_d\cong K[X_1,\ldots, X_d]$ is a polynomial ring in $d$ variables. 
\end{lemma}
\begin{proof}
Write $I= (X_{d+1},X_{d+2},\ldots)$. We will show that $t(I_d)= I$. 
On the one hand, if $n>d$ then $X_n$ corresponds to the trivial representation of $S_n$. The Young diagram of the corresponding partition $(n)$ has $n$ columns. Therefore, $X_n\in t(I_d)$. Conversely, assume that $\{\la\}\in t(I_d)$. Then $\la$ has more than $d$ columns. 
If $\la = (l_1,\ldots,l_t)$ then Zelevinsky showed in Section 4.16. of \cite{Zelevinsky} that 
$$\laa = det(X_{l_i+j-i})_{i,j}.$$
The upper row of the matrix $(X_{l_i+j-i})_{i,j}$ contains the elements $X_{l_1},X_{l_1+1},\ldots X_{l_1+t-1}$. The assumption that $\la$ has more than $d$ rows means that $l_1>d$. This implies that all the elements in the first row of the matrix are in $I$, and the determinant of the matrix is thus also in $I$.
This shows that $t(I_d)\subseteq I$ and we are done.
\end{proof}
This example shows that by picking the ``right'' basis for $\Ainf$ for which we get a PSH-algebra structure, not only a rational PSH-algebra structure, we can get a very clean description for the ideal $I_d$. This leads to the following question:
\begin{question}\label{qu:repth}
Let $\Ainf$ be the universal invariant ring arising from an algebraic structure of type $((p_1,q_1),\ldots, (p_r,q_r))$. Is there a natural $\Z$-lattice $A\subseteq \Ainf$ such that $A$ is a PSH-algebra with respect to the operations induced by those in $\Ainf$? Does this basis gives a neat description of the ideals $I_d$? Can we describe $\Ainf$ in representation-theoretic terms using that basis?
\end{question}

\section*{Acknowledgments}
I would like to thank L{\'o}r{\'a}nt Szegedy for his help with the tikzit package.

\end{document}